\pdfoutput=1 
\documentclass[a4paper]{amsart}

\usepackage[a4paper, margin=1.5in]{geometry}
\usepackage{amsmath,amstext,amssymb,mathrsfs,amscd,amsthm,indentfirst}
\usepackage{amsfonts}
\usepackage{mathtools}
\usepackage{stmaryrd}
\usepackage{booktabs}
\usepackage{verbatim}
\usepackage{enumerate}

\usepackage{graphicx,import}
\graphicspath{{Images/}}
\numberwithin{equation}{section}
\usepackage{tabularx}
\usepackage[shortlabels]{enumitem}
\usepackage{bm}
\usepackage{amsmath,tikz-cd}
\usepackage{dot2texi}
\usetikzlibrary{positioning}
\usepackage{adjustbox}
\usetikzlibrary{arrows.meta}
\usetikzlibrary{arrows,shapes,decorations,automata,backgrounds,petri}
\usepackage{transparent}
\usepackage{graphicx,import}
\usepackage{caption}
\usepackage{subcaption}
\graphicspath{{Images/}}
\usepackage{enumitem}

\newcommand{\mm}{\mathbb}

\newcommand{\diff}{\mathrm{Diff}_{\partial}}
\newcommand{\emb}{\mathrm{Emb}}
\newcommand{\mmm}{\mathcal}

\newcommand{\Emb}{\mathrm{Emb}}

\newcommand{\GL}{\mathrm{GL}}

\newcommand{\SO}{\mathrm{SO}}
\newcommand{\MTSO}{\mathrm{MTSO}}
\newcommand{\GLt}{\GL_2^+(\R)}
\newcommand{\GLd}{\GL_d(\R)}
\newcommand{\GLdp}{\GL_d^+(\R)}

\renewcommand{\:}{\mathop:}

\newcommand{\Tub}{D}

\newcommand{\dcps}{{\tau_\Sigma\times\varepsilon_\Sigma}}
\renewcommand{\int}{\mathrm{int}}

\newcommand{\pushout}[1]{\underset{#1}{\times}}
\newcommand{\pointedpushout}[1]{\underset{#1}{\wedge}}

\newcommand{\parallelsum}{\mathbin{\!/\mkern-5mu/\!}}
\newcommand{\dcup}[1]{{\raisebox{.5ex}{$\scriptstyle \coprod\limits_{#1}$}}}

\newcommand{\fmconf}{\overline{C}}

\mathchardef\ordinarycolon\mathcode`\:
\mathcode`\:=\string"8000
\begingroup \catcode`\:=\active
  \gdef:{\mathrel{\mathop\ordinarycolon}}
\endgroup

\usepackage[all]{xy}
\CompileMatrices

\usepackage{amsthm}
\theoremstyle{plain}
\newtheorem{MainThm}{Theorem}
\newtheorem{MainCor}[MainThm]{Corollary}

\newtheorem{theorem}{Theorem}[section]

\newtheorem{proposition}[theorem]{Proposition}
\newtheorem{lemma}[theorem]{Lemma}

\newtheorem{corollary}[theorem]{Corollary}

\theoremstyle{definition}

\newtheorem{definition}[theorem]{Definition}

\newtheorem{example}[theorem]{Example}

\theoremstyle{remark}

\newtheorem{remark}[theorem]{Remark}
\newtheorem*{remark*}{Remark}

\newcommand{\R}{\mathbb{R}}

\newcommand{\F}{\mathcal{F}}

\newcommand{\id}{\operatorname{id}}

\newcommand{\Map}{\operatorname{Map}}

\newcommand{\Ima}{\operatorname{Im}}

\newcommand{\Top}{\mathrm{Top}}

\newcommand{\Fgb}{F_{g,b}}
\newcommand{\Fgo}{F_{g,1}}

\newcommand{\Fgbk}{F_{g,b}^k}
\newcommand{\Finf}{F_{\infty}}
\newcommand{\Rinf}{\R^{\infty}}

\newcommand{\Csum}{C_\Sigma}
\newcommand{\Dsum}{D_\Sigma}

\newcommand{\Tsum}{D_\Sigma}

\newcommand{\diffM}{{\diff(M)}}

\newcommand{\diffFgo}{\diff(F_{g,1})}

\newcommand{\diffFgbk}{\diff(F_{g,b}^k)}
\newcommand{\diffFgb}{\diff(F_{g,b})}

\newcommand{\diffFinf}{\diff(F_\infty)}

\DeclareMathOperator\colim{colim}

\renewcommand{\paragraph}[1]{\vspace{4pt}\noindent\textbf{{#1}.}}

\title[Decoupling generalised configuration spaces on surfaces]{Decoupling generalised configuration spaces on surfaces}

\author{Luciana Basualdo Bonatto}
\email{basualdo@mpim-bonn.mpg.de}
\address{Max Planck Institute for Mathematics, Vivatsgasse 7, 53111 Bonn, Germany}

\date{\today}
\thanks{This material is based upon work supported by CNPq (201780/2017-8) and by the NSF Grant No. DMS-1928930 while the author participated in a program hosted by the MSRI in 2022.}

\begin{document}
\newpage

\begin{abstract}
The configuration space of $k$ points on a manifold carries an action of its diffeomorphism group. The homotopy quotient of this action is equivalent to the classifying space of diffeomorphisms of a punctured manifold, and therefore admits  results about homological stability. Inspired by the works of Segal, McDuff, Bodigheimer, and Salvatore, we look at generalised configuration spaces where particles have labels and even partially summable labels, in which points are allowed to collide whenever their labels are summable. These generalised configuration spaces also admit actions of the diffeomorphism group and we look at their homotopy quotients. Our main result is a decoupling theorem for these homotopy quotients on surfaces: in a range, their homology is completely described by the product of the moduli space of surfaces and a generalised configuration space of points in $\mathbb{R}^\infty$. 
Using this result, we show these spaces admit homological stability with respect to increasing the genus, and we identify the stable homology. This can be interpreted as an Diff-equivariant homological stability for factorization homology. In addition, we use this result to study the group completion of the monoid of moduli spaces of configurations on surfaces.
\end{abstract}
\maketitle

\vspace{-0.5ex}\section{Introduction}

The \emph{ordered configuration space} of $k$ points on a smooth manifold $M$ without boundary is defined as
    \[\widetilde{C}_k(M)\coloneqq\{(m_1,\dots,m_k)\in M^k\,\mid \,m_i\neq m_j \text{ if }i\neq j\}.\]
When $M$ is a smooth manifold with boundary, we denote by $\widetilde{C}_k(M)$ the space of ordered configurations of $k$ points in its interior. The symmetric group $\Sigma_k$ acts on this space by permuting the order of the $k$ points. The \emph{configuration space} of $k$ points on $M$, denoted $C_k(M)$, is the quotient $\widetilde{C}_k(M)/\Sigma_k$. We denote by $\diff(M)$ the group of diffeomorphisms of a manifold $M$ which fix a collar of its boundary. 

In this paper, we will focus on $2$-dimensional manifolds and we denote by $\Fgbk$ an orientable surface of genus $g$, $k$ punctures, and $b\geq 1$ boundary components.  The spaces $C_k(\Fgb)$ admit an action of the group $\diffFgb$, where a diffeomorphism $\phi$ acts by taking a collection of $k$ points to its image via $\phi$. Of interest here, is the Borel construction (homotopy quotient) of this action, denoted $C_k(\Fgb)\parallelsum \diffFgb$, to which we refer to as a \emph{moduli of configurations of $k$ points in $\Fgb$}. It is simple to show that
    \begin{align}\label{eq: relation of moduli of configurations to classifying space of punctured surfaces}
        C_k(\Fgb)\parallelsum \diffFgb\simeq B\diffFgbk
    \end{align}
and in fact this relation is not only true for surfaces, but for any manifold with $k$ punctures. In particular, this allows us to deduce homological stability results for these moduli of configurations of $k$ points, directly from the known theorems for classifying spaces of punctured surfaces. For instance, when $b\geq 1$ these spaces admit homological stability when increasing the genus and when increasing then number of points \cite{MR786348,MR1054572,RWMR3438379,MR3180616}. Moreover, it was shown in \cite{MR1851247} that the stable homology of this classifying space can be computed from the homology of $B\diffFgb\times B(\Sigma_k\wr \GLt)$, what is known as a decoupling theorem.

In this paper, we study the analogue of this moduli space for generalised configuration spaces. As a first case, we look at labelled configurations: given a pointed space $Z$, the space of \emph{configurations in $M$ with labels in $Z$} is the quotient 
    \[C(M;Z)\coloneqq\left(\coprod\limits_{k\geq 0}\widetilde{C}_k(M)\times_{\Sigma_k}Z^k\right)\bigg/\sim\]
under the relation $(m_1,\dots,m_k;z_1,\dots,z_k)\sim (m_1,\dots,m_{k-1};z_1,\dots,z_{k-1})$ if $z_k$ is the basepoint of $Z$. We can interpret this space geometrically by considering it as the space of particles in $M$ where each particle is labelled by an element of $Z$, and a particle is allowed to disappear if labelled by the basepoint. 

This space has been of interest since the 70's, appearing on the seminal works of May \cite{May1972} and Segal \cite{Segal}. It was noted that the space $C(\R^n;Z)$ can be given the structure of an ($A_\infty$-)monoid by taking multiplication to be roughly given by stacking configurations side by side \cite{Segal}. One of the main results about this space is what today is called a \emph{scanning map} $C(\R^n;Z)\to \Omega^n\Sigma^n Z$ which was shown by Segal to induce a weak-homotopy equivalence on group-completions.
This idea has been generalised in many directions. For instance, B\"odigheimer \cite{bodigheimer1987stable} proved an analogous statement for configurations on general manifolds. In addition, similar results were proven for the case where the spaces of labels has extra structure, such as (partial) multiplications \cite{mcduff1975configuration,Segal79,Guest95,Kallel01,Salvatore1999Configuration}. We discuss the later case in Section \ref{subsec: intro to partially summable configs}.

More recently this labelled configuration space and scanning map argument have been expanded to sophisticated constructions in factorization homology \cite{AyalaFrancisTanaka17} on the one hand and, on the other, in the form of configuration spaces of manifolds, has been used to compute the stable homology of the moduli spaces of Riemann surfaces \cite{madsen2007stable} and higher dimensional manifolds \cite{MR3718454,MR3665002}.

Labelled configuration spaces also inherit an action of the diffeomorphism group. Even more, if $Z$ is a pointed $\GLt$-space, we can define an action of $\diffFgb$ on $C(M;Z)$  where a diffeomorphism $\phi$ acts by taking a collection of $k$ points to its image via $\phi$, and the label $z$ of a point $w$ is taken to the label $d_w\phi\cdot z$ of the point $\phi(w)$. Unlike the case of configurations with a fixed number of points, $C(\Fgb;Z)\parallelsum \diffFgb$ is not in general equivalent to the classifying space of a diffeomorphism group. Hence we ask if it still has homological stability and if it admits an analogous decoupling theorem.

For a surface $\Fgb$, with $b\geq1$, taking the boundary connected sum with the surface $F_{1,1}$ induces a homomorphism $\diffFgb\to \diff(F_{g+1,b})$, given by extending a map on $\Fgb$ by the identity. We call this the stabilisation map and study what it induces on the moduli of configuration spaces:

\begin{MainThm}\label{mainthm: homological stability for moduli of labelled configurations}
    Let $Z$ be a pointed $\GLt$-space and $b\geq 1$. The stabilisation map on the Borel constructions
        \[s_*:C(\Fgb;Z)\parallelsum\diffFgb\to C(F_{g+1,b};Z)\parallelsum\diff(F_{g+1,b})\]
    induces a homology isomorphism in degrees $\leq \frac{2}{3}g$.
\end{MainThm}

Moreover, we can determine precisely what the stable homology is:

\begin{MainThm}\label{mainthm: the stable homology for labelled configurations}
    Let $Z$ is a pointed connected $\GLt$-space. There is a map
        \[C(\Fgb;Z)\parallelsum\diffFgb\to \Omega^\infty \MTSO(2)\times \Omega^\infty \Sigma^\infty\left(E\GLt_+\pointedpushout{\GLt} Z\right)\]
    which is compatible with the stabilisation maps and induces a homology isomorphism in degrees $\leq \frac{2}{3}g$.
\end{MainThm}

In the above, $E\GLt$ denotes the total space of a universal fibration for $B\GLt$, we use $(-)_+$ to denote adjoining a disjoint basepoint to a space, and $-\pointedpushout{\GLt} -$ denotes the quotient of the smash product of pointed topological $\GLt$-spaces by the diagonal action of $\GLt$.

Both of the results above are consequences of Theorem \ref{mainthm: decoupling labelled configurations} below, which is an analogue of the decoupling theorem in \cite{MR1851247}. It implies that the stable homology of this moduli of configuration spaces can be understood through a decoupling map $\tau\times \varepsilon$, which separates the points in the configurations from the underlying surfaces. The map $\tau: C_k(\Fgb)\parallelsum \diffFgb\to B\diffFgb$ forgets the data of the configurations, and $\varepsilon$ forgets the underlying surface, but still remembers the points in the configuration and some local tangential data around them (for a detailed description of these maps see Section \ref{chap: configuration spaces}). 

\begin{MainThm}[Decoupling Theorem for Labelled Configurations]\label{mainthm: decoupling labelled configurations}
    Let $\tau$ and $\varepsilon$ be the maps described above. Then 
        \begin{align*}
            \tau\times\varepsilon:C(\Fgb;Z)\parallelsum \diffFgb \to B\diffFgb \times C(\mm{R}^\infty;E\GLt_+\pointedpushout{\GLt} Z)
        \end{align*}
    induces a homology isomorphism in degrees $\leq \frac{2}{3}g$.
\end{MainThm}

This result may be interpreted in physical terms: in the high genus limit, the constraints for the particles to stay on the underlying surface are lifted and the particles are now free. 

As a final application of Theorem \ref{mainthm: decoupling labelled configurations}, we look at monoids of moduli of configurations on surfaces. Boundary connected sum induces a multiplication on the level of classifying spaces 
making $\dcup{g}B\diffFgo$ into a topological monoid. This important construction and its group completion have been central in the study of the stable homology of mapping class groups of surfaces \cite{Miller86,Tillmann_surface_operad,madsen2007stable,MR2653727,MR3718454,MR3665002}. This  gluing of surfaces induces also a multiplication on the Borel constructions
    \[C(\Fgo;Z)\parallelsum\diffFgo\times C(F_{h,1};Z)\parallelsum\diff(F_{h,1})\to C(F_{g+h,1};Z)\parallelsum\diff(F_{g+h,1})\]
making $\dcup{g}C(\Fgo;Z)\parallelsum\diffFgo$ into a topological monoid. We study its group completion.

\begin{MainCor}\label{mainthm: group completion}
   For any pointed $\GLt$-space $Z$, the decoupling map induces a weak equivalences on group completions
        \[\Omega B\left(\dcup{g}C(\Fgo;Z)\parallelsum\diffFgo\right) \simeq \Omega B\left(\dcup{g}B\diffFgo\right)\times \Omega BC(\Rinf;E\GLt_+\pointedpushout{\GLt}Z).\]
\end{MainCor}

\subsection{Configuration spaces with partially summable labels}\label{subsec: intro to partially summable configs}

The main result of this paper considers a more general type of configuration spaces with labels in a framed partial $2$-monoid, where particles are allowed to collide if their labels are summable. This space has been explored in works such as \cite{mcduff1975configuration,Segal79,Guest95,Kallel01,Salvatore1999Configuration} and, when the labels are $E_2$-algebras (not partial), is equivalent to factorization homology \cite{MR3431668} and topological chiral homology \cite{MR2555928}. The first example of this construction can be seen in McDuff's configuration spaces of positive and negative particles \cite{mcduff1975configuration}, where particles are labelled by ``charges" and are allowed to collide whenever their charges are opposite. More generally, Salvatore \cite{Salvatore1999Configuration} defines partial $2$-monoids, which are, in essence, spaces with a multiplication similar to an $E_2$-algebra structure, but with the restriction that this multiplication does not need to be defined for every tuple of elements (see Definition \ref{def:partial d-monoid}). Whenever $P$ is equipped with a compatible action of $\GLt$, we call it a framed partial monoid, and we can define the space of configurations in $\Fgb$ with partially summable labels in $P$, denoted $\Csum(\Fgb;P)$. The definition $\Csum(\Fgb;P)$ requires much more machinery then the case for non-summable labels, such as the Fulton-MacPherson operad, and yields more complicated models for configuration spaces. We discuss these constructions in Section \ref{sec:partial d-monoids}. As before, these generalised configuration spaces admit an action of the diffeomorphism group and we study its Borel construction.

\begin{MainThm}\label{mainthm: homological stability for moduli of summable configurations}
    Let $P$ be a framed partial $2$-monoid and $b\geq 1$. The stabilisation map on the Borel constructions
        \[s_*:\Csum(\Fgb;P)\parallelsum\diffFgb\to \Csum(F_{g+1,b};P)\parallelsum\diff(F_{g+1,b})\]
    induces a homology isomorphism in degrees $\leq \frac{2}{3}g$.
\end{MainThm}

While homological stability for configurations with summable labels with respect to increasing the number of points had been studied in \cite{MR3464033}, the above result is the first to look at stability with respect to increasing the genus. Theorem \ref{mainthm: homological stability for moduli of summable configurations} can be interpreted as a $\diff$-equivariant homological stability result for factorisation homology.

As in the case for labelled configurations, this result is a consequence of a decoupling theorem for the space $\Csum(\Fgb;P)\parallelsum\diffFgb$, which is the main result of this paper. This is much more intricate than the decoupling for labelled configurations. The proof uses a semi-simplicial resolution of $\Csum(\Fgb;P)$ developed in section \ref{sec:poset-surrounded-config}, which we refer to as the \emph{disc model for configurations}, denoted $|\Tsum(\Fgb;P)_\bullet|$ (Proposition \ref{prop: weak equivalence with poset one}). This model makes explicit the connection between these spaces and factorization homology. In the decoupling context, we naturally encounter an analogue of this space with $2$-dimensional discs with configurations embedded in $\Rinf$, we denote this space $|\Tsum^2(\Rinf;P)_\bullet|$ (Definition \ref{def:thin-discs}). Using the Decoupling Theorem for Labelled Configurations (Theorem \ref{mainthm: group completion}) we then prove:

\begin{MainThm}[Decoupling Theorem for Configurations with Summable Labels]\label{main thm: decoupling summable labels}
    For $P$ a framed partial $2$-monoid, there is a weak equivalence $\Csum(\Fgb;P)\parallelsum\diffFgb\simeq |\Tsum(\Fgb;P)_\bullet|\parallelsum\diffFgb$ and the decoupling map 
        \[|\Tsum(\Fgb;P)_\bullet|\parallelsum\diffFgb\to B\diffFgb\times |\Tsum^2(\Rinf;P)_\bullet|.\]
    induces a homology isomorphism in degrees $\leq\frac{2}{3}g$.
\end{MainThm}

In future work we will discuss the homotopy type of the space $|\Tsum^2(\Rinf;P)_\bullet|$ and its description as an infinite loop space. We conjecture that it is also equivalent to a configuration in $\Rinf$ with partially summable labels. 

Analogous to the case of labelled configurations, the spaces $\Csum(\Fgo;P)\parallelsum\diffFgo$ assemble into a topological monoid, and the decoupling theorem descends into its group completion.

\begin{MainCor}\label{mainthm: group completion summable}
    For any path-connected framed partial $2$-monoid with unit $P$, the decoupling map induce a homotopy equivalence
        \[\Omega B(\dcup{g}\Csum(\Fgo;P)\parallelsum\diffFgo)\simeq \Omega B(\dcup{g}B\diffFgo)\times 
        \Omega B(|\Tsum^2(\Rinf;P)_\bullet|).\]
\end{MainCor}

\subsection{Outline of the paper}
We start by recalling in Section \ref{section:preliminaries} background results which will be used throughout the paper, especially on Section \ref{chap: configuration spaces summable labels}. This can be skipped and referred back to when necessary.

In Section \ref{chap: configuration spaces} we introduce labelled configuration spaces and prove Theorem \ref{mainthm: decoupling labelled configurations}. Using this, we deduce Theorems \ref{mainthm: homological stability for moduli of labelled configurations} and \ref{mainthm: the stable homology for labelled configurations}, and Corollary \ref{mainthm: group completion}.

In Section \ref{chap: configuration spaces summable labels} we discuss the case of configurations with summable labels, and prove the main results of the paper. We start by recalling in \ref{sec:partial d-monoids} the definitions of framed partial $d$-monoids and configuration spaces with partially summable labels. We then construct semi-simplicial resolutions for these spaces in Section \ref{sec:poset-surrounded-config}. In Section \ref{sec:decoupling map for colliding configurations}, we use this disc model together with the Decoupling Theorem for Labelled Configurations (Theorem \ref{mainthm: group completion}) to prove Theorem \ref{main thm: decoupling summable labels}. Finally, we use this result to deduce Theorem \ref{mainthm: homological stability for moduli of summable configurations} and Corollary \ref{mainthm: group completion summable}.

\subsection*{Acknowledgements} I would like to thank Ulrike Tillmann for suggesting the problem and for the many insightful conversations. In addition, I would like to thank  David Ayala, Christopher Douglas, Jan Steinebrunner, and Nathalie Wahl for the helpful discussions and comments. 

\section{Preliminaries}\label{section:preliminaries}

In this section we recall techniques and results on semi-simplicial spaces, which will be used in Section \ref{chap: configuration spaces summable labels}. The reader may skip this part and refer back to when necessary. For a detailed exposition of the concepts in this section see \cite{Ebert_2019}.

A \emph{semi-simplicial space} is a functor $\Delta_{inj}^{op}\to \Top$, where $\Delta_{inj}$ is the category with object the linearly ordered sets $[p]=\{0<\dots<p\}$ and morphisms the injective monotone maps. We denote such a functor by $X_\bullet$ and write $X_p=X_\bullet(\{1,\dots,p\})$. The datum of a semi-simplicial space is equivalent to the collection of spaces $X_p$, $p\geq 0$, together with \emph{face maps} $d_i:X_p\to X_{p-1}$ for $i=0,\dots,p$, satisfying $d_id_j=d_{j-1}d_i$ if $i<j$.

Denote by $\Delta^p$ the \emph{standard $p$-simplex}
    \[\Delta^p=\left\{(t_0,\dots,t_p)\in\R^{p+1}\Big|\sum\limits_{i=1}^p t_i=1 \text{ and } t_i\geq0 \text{ for all }i\right\}.\]
To each morphism $\phi:[p]\to[q]$ in $\Delta_{inj}$, there is a continuous map $\phi_*:\Delta^p\to \Delta^q$ such that $\phi_*(t_0,\dots,t_p)=(s_0,\dots,s_q)$ with $s_j=\sum_{i\in\phi^{-1}(j)}t_i$.
The \emph{geometric realisation} of a semi-simplicial space $X_\bullet$ is the quotient space 
    \[|X_\bullet|\coloneqq \left(\coprod\limits_{p}X_p\times\Delta^p\right)\bigg/\sim\]
where $(x,\phi_* t)\sim (\phi^* x,t)$, and $\phi$ is a morphism of $\Delta_{inj}$.

\subsection{Semi-simplicial nerve of a poset}\label{subsec: semi-simplicial nerve of a poset}
Any topological poset $(Q,<)$ defines a semi-simplicial space $Q_\bullet$ by setting $Q_p$ to be the subspace of tuples $(q_0<\dots<q_p)\in Q^{p+1}$, and face maps 
    \begin{align*}
        d_i:Q_p &\longrightarrow Q_{p-1} & \text{for }0\leq i\leq p\\
        (q_0<\dots<{q_i}<\dots<q_p)&\longmapsto (q_0<\dots<q_{i-1}<q_{i+1}<\dots<q_p).&
    \end{align*}
We refer to $Q_\bullet$ as the \emph{semi-simplicial nerve} of the poset $Q$.

Given a topological poset $(Q,<)$, the space $Q\times Q$ can be equipped with a partial order where $(q_1,q_2)<(q_1',q_2')$ if $q_i<q_i'$ and $q_j\leq q_j'$, for $\{i,j\}=\{1,2\}$. We say that a pointed such $Q$ is a \emph{partially ordered topological monoid} if it is equipped with a multiplication
    \[-\cdot-:Q\times Q\to Q\]
which is strictly associative, unital and order preserving. In this case, the geometric realisation of the semi-simplicial nerve $Q_\bullet$ is naturally endowed with a multiplication $\cdot$ defined by
        \begin{align*}
            \big((q_0<\dots<q_m;t_0,\dots,t_m)\cdot (\overline{q_0}<\dots<\overline{q_k};\overline{t}_0,\dots,\overline{t}_k)\big) = \\
            = (q_0\cdot \overline{q_0}<\dots<q_0\cdot \overline{q_k}<\dots<q_m\cdot \overline{q_0}\dots<q_m\cdot \overline{q_k};
            t_0\cdot \overline{t},\dots,t_m \cdot \overline{t})
        \end{align*}
where  $t_i\cdot \overline{t}=t_i\overline{t_0},\dots,t_i\overline{t_k}$, for all $i=0,\dots,m$. It is straightforward to verify that this is well-defined.

\begin{lemma}\label{lemma: monoid on geom realisation}
    For $(Q,<,\mu)$ a partially ordered topological monoid,  $(|Q_\bullet|,|\mu|)$ is a topological monoid. Moreover, any map of partially ordered topological monoids $f:Q\to Q'$ induces a map of topological monoids
        \[f_*:|Q_\bullet|\to |Q_\bullet'|.\]
\end{lemma}

The proof is a straightforward computation and follows directly from the definitions.

\subsection{Spectral sequence}\label{subseq: spectral sequence}
We quickly recall below the spectral sequence defined in \cite[Proposition 5.1]{segal1968classifying} associated to a semi-simplicial space, which is the key for the homology argument used in the proof of Theorem \ref{thm: decoupling summable configurations} (see \cite[Section 1.4]{Ebert_2019} for a detailed discussion).

For any semi-simplicial space $X_\bullet$, the geometric realisation $|X_\bullet|$ admits a filtration by its skeleta, with $|X_\bullet|^{(0)}=X_0$ and
    \[|X_\bullet|^{(q)}=|X_\bullet|^{(q-1)}\cup_{X_q\times\partial\Delta^q} X_q\times \Delta^q.\]
This filtration yields a spectral sequence 
    \[E^1_{p,q}=H_{p+q}(|X_\bullet|^{(q)},|X_\bullet|^{(q-1)})\Longrightarrow H_{p+q}(|X_\bullet|)\]
and by excision and the Kunneth isomorphism, the left-hand term can be re-written to give a spectral sequence with 
    \[E^1_{p,q}\cong H_{p}(X_q)\Longrightarrow H_{p+q}(|X_\bullet|).\]
Therefore a map of semi-simplicial spaces inducing a level-wise homology isomorphism gives an isomorphism of the first pages of the respective spectral sequences, and therefore a homology isomorphism between the geometric realisations.

\subsection{Semi-simplicial Resolutions}

In the proof of the decoupling we will use a semi-simplicial resolution of the spaces of configurations with summable labels. Showing that we indeed have a resolution will be a direct application of \cite[Theorem 6.2]{MR3207759}. For completeness, we quickly recall the statement of this result to clarify the conditions that will be checked in the proof of Proposition \ref{prop: weak equivalence with poset one}. We follow the notation and definitions of \cite{MR3207759}.

An \emph{augmented semi-simplicial space}\index{semi-simplicial space!augmented} is a triple  $(X_\bullet,X_{-1},\varepsilon_\bullet)$, where $X_\bullet$ is a semi-simplicial space, $X_{-1}$ is a space and $\varepsilon_\bullet$ is a collection of continuous maps $\varepsilon_p:X_p\to X_{-1}$ satisfying $d_i\varepsilon_p=\varepsilon_{p-1}$ for all $p\geq 0$ and all face maps $d_i$. We also say that $\varepsilon_\bullet:X_\bullet\to X_{-1}$  is an \emph{augmentation} for $X_\bullet$. It is simple to verify that an augmentation induces a continuous map $|\varepsilon_\bullet|:|X_\bullet|\to X_{-1}$.

\begin{definition}\label{def: topological flag complex}
    An \emph{augmented topological flag complex} \cite[Definition 6.1]{MR3207759} is an augmented semi-simplicial space $\varepsilon:X_\bullet\to X_{-1}$ such that
    \begin{enumerate}[(i)]
        \item The map $X_n\to X_0\times_{X_{-1}}\dots \times_{X_{-1}} X_0$ taking an $n$-simplex to its $(n+1)$ vertices is a homeomorphism onto its image, which is an open subset.
        \item A tuple $(v_0,\dots,v_n)\in X_0\times_{X_{-1}}\dots \times_{X_{-1}} X_0$ lies in $X_n$ if and only if $(v_i,v_j)\in X_1$ for all $i<j$.
    \end{enumerate}
\end{definition}
    
In other words, in an augmented topological flag complex, the space of $n$-simplices can be described as an open subspace of the $(n+1)$-tuples of vertices with the same image under $\varepsilon$, and such a tuple forms an $n$-simplex if and only if the pairs of vertices are all $1$-simplices. The result below is a criterion to determine when an augmented topological flag complex $X_\bullet\to X_{-1}$ induces a weak equivalence $|X_\bullet|\to X_{-1}$.

\begin{theorem}[\cite{MR3207759}, Theorem 6.2]\label{thm:augmented semi-simplicial spaces and realisation}
    Let $X_\bullet\to X_{-1}$ be an augmented topological flag complex. Suppose that
        \begin{enumerate}[(i)]
            \item The map $\varepsilon:X_0\to X_{-1}$ has local lifts of any map from a disc, i.e. given a map $f:D^n\to X_{-1}$, a point $p\in\varepsilon^{-1}(f(x))$, there is an open neighbourhood $U\subset D^n$ of $x$ and a map $F:U\to X_0$ such that $\varepsilon\circ F=f|_U$ and $F(x)=p$.
            \item $\varepsilon:X_0\to X_{-1}$ is surjective.
            \item For any $p\in X_{-1}$ and any non-empty finite set $\{v_1,\dots,v_n\}\in \varepsilon^{-1}(p)$ there exists a $v\in\varepsilon^{-1}(p)$ with $(v_1,v)\in X_{1}$ for all $i$.
        \end{enumerate}  
    Then $|X_\bullet|\to X_{-1}$ is a weak homotopy equivalence.
\end{theorem} 
\section{Decoupling Labelled Configuration Spaces}\label{chap: configuration spaces}

In this section, we recall the definition of configuration spaces with labels and introduce a model for the homotopy quotients $C(\Fgb;Z)\parallelsum \diff(\Fgb)$. With this, we construct the decoupling map of Theorem \ref{mainthm: decoupling labelled configurations}. We then prove this result as Theorem \ref{thm: decoupling labelled configurations} and use it to prove Theorems \ref{mainthm: homological stability for moduli of labelled configurations} and \ref{mainthm: the stable homology for labelled configurations}, and Corollary \ref{mainthm: group completion}.
These are, respectively, Corollary \ref{cor: homological stability for moduli of labelled configurations}, Corollary \ref{cor: homology iso on limit spaces - labelled}, and Corollary \ref{thm:group-completion} below.

\begin{definition}
    Let $M$ be a manifold and $Z$ be a well-pointed space.
    The \emph{configuration space of $M$ with labels in $Z$}, denoted $C(M;Z)$, is the quotient
        \[\dcup{k\geq 0} \widetilde{C}_k(M)\pushout{\Sigma_k}Z^k/\sim\]
    where $\widetilde{C}_k(M)$ denotes the ordered configuration space of $k$ points in $M$, and 
        \[(m_1,\dots,m_k;z_1,\dots,z_k)\sim(m_1,\dots,m_{k-1};z_1,\dots,z_{k-1})\]
    whenever $z_k$ is the basepoint of $Z$. 
\end{definition}

If $Z$ is a pointed $\GLt$-space, and $M$ is a surface $\Fgb$, with $b\geq 1$, the space $C(\Fgb;Z)$ carries a natural action by the diffeomorphism group of $\Fgb$: for $\phi\in \diffFgb$
    \[\phi\cdot (m_1,\dots,m_k;z_1,\dots,z_k)\coloneqq(\phi(m_1),\dots,\phi(m_k);D_{m_1}\phi\cdot z_1,\dots,D_{m_1}\phi\cdot z_k).\]
The basepoint relation is preserved by this action as $Z$ is a pointed $\GLt$-space. The decoupling result is about the homotopy quotient of this action, that is, the Borel construction $C(\Fgb;Z)\parallelsum \diffFgb$. 

From now on, we denote by $\Fgb$ a fixed orientable surface of genus $g$, $b\geq 1$ boundary components, and pick once and for all a framing on $\Fgb$, that is, a section $s$ of the frame bundle $Fr(T\Fgb)\to \Fgb$. We fix now a model for $E\diffFgb$ that will be used to construct the decoupling map of Theorem \ref{thm: decoupling labelled configurations}. Let $\emb(\Fgb,\Rinf)$ denote $\colim\limits_{n\to \infty}\emb(\Fgb,\R^n)$. This space has a free action of $\diffFgb$ by precomposition and by \cite{MR613004} the quotient map $\emb(\Fgb,\Rinf)\to\emb(\Fgb,\Rinf)/\diffFgb$ has slices, hence it is a principal $\diffFgb$-bundle. Moreover, by Whitney's embedding theorem, $\emb(\Fgb,\Rinf)$ is weakly contractible and therefore it is a model for $E\diffFgb$. We then let
    \[C(\Fgb;Z)\parallelsum \diffFgb \simeq \emb(\Fgb,\Rinf)\underset{\diff(\Fgb)}{\times} C(\Fgb;Z).\]
Analogously, we will take the model for $B\diffFgb$ given by 
    \[B\diffFgb \simeq \emb(\Fgb,\Rinf)/\diff(\Fgb).\]
This can be interpreted as the space of abstract submanifolds of $\Rinf$ which are diffeomorphic to $\Fgb$, but without a fixed diffeomorphism. Analogously, a point in $C(\Fgb;Z)\parallelsum \diffFgb$ consists of one such abstract manifold, together with a labelled configuration.

The decoupling map will be the product of two maps 
    \begin{align}\label{eq: forgetful and evaluation maps}
        \tau:C(\Fgb;Z)\parallelsum \diffFgb &\to B\diffFgb\\
        \varepsilon: C(\Fgb;Z)\parallelsum \diffFgb &\to C(\mm{R}^\infty;E\GLt_+\pointedpushout{\GLt} Z).
    \end{align}
Recall that $E\GLt$ is the total space of a universal fibration for $B\GLt$, we use $(-)_+$ to denote adjoining a disjoint basepoint to a space, and $-\pointedpushout{\GLt} -$ denotes the quotient of the smash product of pointed topological $\GLt$-spaces by the diagonal action of $\GLt$.

Intuitively, the map $\tau$ forgets the configuration, while $\varepsilon$ forgets the underlying surface, but remembers the labelled configuration together with data on their tangent space on the submanifold.

In details, $\tau$ is the classifying map for the homotopy quotient, and it is simply the one induced by the projection $\emb(\Fgb,\Rinf){\times} C_k(\Fgb) \to \emb(\Fgb,\Rinf)$.

To define $\varepsilon$, we take as model for $B\GLt$ the oriented Grassmanian manifold of $2$-dimensional oriented subsbaces of $\Rinf$,  $Gr^+(2,\infty)$, and let $E\GLt$ denote the total space of the universal $\GLt$-bundle over it. Then using the identification $Fr(T\Rinf)\cong \Rinf\times E\GLt$, an embedding $e:\Fgb\hookrightarrow\Rinf$ induces a map $e_*:Fr(T\Fgb)\to ESO(2)$ from the bundle of framings on $T\Fgb$, taking a basis of $T_p\Fgb$ to its image via $D_p e$. The map $\varepsilon$ takes a point represented by a labelled configuration $[m_1,\dots,m_k;z_1,\dots,z_k]$ and an embedding $e:\Fgb\hookrightarrow\Rinf$ to the labelled configuration in $\Rinf$ given by
    \[[e(m_1),\dots,e(m_k);[e_*(s(m_1)),z_1],\dots, [e_*(s(m_k)),z_k]]\]
where $s(p)$ denotes the chosen oriented frame on $p\in\Fgb$. It is simple to verify that this indeed defines a continuous function to the configuration space $C(\mm{R}^\infty;(E\GLt)_+\pointedpushout{\GLt} Z)$.

\begin{theorem}\label{thm: decoupling labelled configurations}
    Let $\tau$ and $\varepsilon$ be the maps in \eqref{eq: forgetful and evaluation maps}. Then the decoupling map
        \begin{align*}
        \tau\times\varepsilon:C(\Fgb;Z)\parallelsum \diffFgb \to B\diffFgb \times C(\mm{R}^\infty;(E\GLt)_+\pointedpushout{\GLt} Z)
        \end{align*}
    induces a homology isomorphism in degrees $\leq \frac{2}{3}g$.
\end{theorem}

The proof of the result above will build upon a decoupling result for unlabelled configurations with a fixed number of points, which was first proved by \cite{MR1851247} and generalised in \cite{MR2439464,bonatto_2022}. We show here a slight generalisation of the result which we will need in the proof. The space $C(M;Z)$ is constructed as a quotient of the union of spaces $\widetilde{C}_k(M)\times_{\Sigma_k}Z^k$, and the Lemma below is about the decoupling map in each of these components. 

In fact, we will work on a slightly more general context which will be more convenient for the proof: let $X$ be a well-pointed space with an action of the wreath product $\Sigma_k\wr \GLt$ (in the context above we were using $X=Z^k$). The space $\widetilde{C}_k(\Fgo)\times X$ comes equipped with two actions: $\Sigma_k$ acts diagonally by permuting the points in the configuration and by the action on $X$, and $\diffFgb$ acts by
    \[\phi\cdot(m_1,\dots,m_k;x)=(\phi(m_1),\dots,\phi(m_k);(d_{m_1}\phi,\dots,d_{m_k}\phi)(x)).\]
Note that the actions of $\Sigma_k$ and $\diffFgo$ on this space commute.

As before, we have maps     
    \begin{align}\label{eq: forgetful and evaluation maps for fixed n}
        \tau_k:(\widetilde{C}_k(\Fgo)\pushout{\Sigma_k} X)\parallelsum \diffFgb &\to B\diffFgb\\
        \varepsilon_k: (\widetilde{C}_k(\Fgo)\pushout{\Sigma_k} X)\parallelsum \diffFgb &\to (\widetilde{C}_k(\mm{R}^\infty)\times (E\GLt)^k)\pushout{\Sigma_k \wr \GLt} X.
    \end{align}
Here $\tau_k$ is again simply the classifying map for the homotopy quotient, and $\varepsilon_k$ is the map taking a point represented by $[m_1,\dots,m_k;x]$ and an embedding $e:\Fgb\hookrightarrow\Rinf$ to the class
    \[\big[ [e(m_1),\dots,e(m_k);e_*(s(m_1)), \dots, e_*(s(m_k))],x \big]\]
where $s(p)$ denotes the chosen oriented frame on $p\in\Fgb$. We remark that replacing $X$ by $Z^k$ one recovers precisely the definition of the maps $\tau$ and $\varepsilon$ in \eqref{eq: forgetful and evaluation maps}.

\begin{lemma}[\cite{MR1851247,bonatto_2022}]\label{lemma:parametrised-decoupling}
    For any $\Sigma_k\wr \GLt$-space $X$, let $\tau_k$ and $\varepsilon_k$ be the maps defined in \eqref{eq: forgetful and evaluation maps for fixed n}. Then 
        \[\tau_k\times \varepsilon_k: (\widetilde{C}_k(\Fgb)\pushout{\Sigma_k}X)\parallelsum \diffFgb\to B\diffFgb\times (\widetilde{C}_k(\mm{R}^\infty)\times (E\GLt)^k)\pushout{\Sigma_k \wr \GLt} X.\]
    induces a homology isomorphism in degrees $\leq \frac{2}{3}g$.
\end{lemma}

For completeness, we include a short proof of the above result. For details see \cite{MR1851247,bonatto_2022}.

\begin{proof}[Proof of Lemma \ref{lemma:parametrised-decoupling}]
    We start by reducing the proof to the case when $X$ is a point. The projections
        \begin{align*}
            (\widetilde{C}_k(\Fgb)\pushout{\Sigma_k}X)\parallelsum \diffFgb &\to {C}_k(\Fgb)\parallelsum \diffFgb\\
            B\diffFgb\times (\widetilde{C}_k(\mm{R}^\infty)\times (E\GLt)^k)\pushout{\Sigma_k \wr \GLt} X &\to  B\diffFgb\times {C}_k(\mm{R}^\infty,B\GLt)
        \end{align*}
    are both fibrations with fibre $X$, and the map $\tau_k\times \varepsilon_k$ induces a map between the corresponding fibre sequences, which is the identity on the fibers. If the map between the base spaces induces a homology isomorphism in degrees $\leq \frac{2}{3}g$, then by Zeeman's Comparison Theorem \cite{Zeeman} applied to the Serre spectral sequences associated to these fibrations, so does $\tau_k\times \varepsilon_k$. Hence it is enough to show that the map between the base spaces
        \[{\tau_k}\times {\varepsilon_k}:{C}_k(\Fgb)\parallelsum \diffFgb\to B\diffFgb\times {C}_k(\mm{R}^\infty,B\GLt)\]
    induces a homology isomorphism in degrees $\leq \frac{2}{3}g$.

    By Palais' Theorem \cite{MR117741}, the map $\varepsilon_k$ is a fibration with fiber $B\diff(F_{g,b+n})$. Moreover, the map ${\tau_k}\times {\varepsilon_k}$ induces a map of fibre sequences
        \[\begin{tikzcd}
        	{B\diff(F_{g,b+n})} & {{C}_k(\Fgb)\parallelsum \diffFgb} & {C_k(\mm{R}^\infty,B\GLt)} \\
        	B\diffFgb & {B\diffFgb \times C_k(\mm{R}^\infty,B\GLt)} & {C_k(\mm{R}^\infty,B\GLt)}
        	\arrow[from=2-1, to=2-2]
        	\arrow[from=2-2, to=2-3, "proj"]
        	\arrow[from=1-3, to=2-3, equal]
        	\arrow[from=1-2, to=2-2, "{\tau_k}\times {\varepsilon_k}"]
        	\arrow[from=1-1, to=2-1]
        	\arrow[from=1-1, to=1-2]
        	\arrow[from=1-2, to=1-3, "{\varepsilon_k}"]
        \end{tikzcd}\]
    where the leftmost map is the one induced by capping off the $n$ extra boundary components by gluing discs. This map was shown to induce a homology isomorphism in the range $\leq\frac{2}{3}g$ \cite{MR786348,ivaMR896878,ivaMR1015128,ivaMR1234264,boldsen2012improved,RWMR3438379}. Hence, by Zeeman's Comparison Theorem \cite{Zeeman} applied to the Serre spectral sequences associated to these fibre sequences, the map between the total spaces also induces homology isomorphisms in the range $\leq\frac{2}{3}g$.
\end{proof}

Equipped with Lemma \ref{lemma:parametrised-decoupling}, we are now ready to prove Theorem \ref{thm: decoupling labelled configurations}.

\begin{proof}[Proof of Theorem \ref{thm: decoupling labelled configurations}]
    The spaces $C(\Fgb;Z)$ and $C(\mm{R}^\infty;(E\GLt)_+\wedge_{\GLt} Z)$ consist of configuration with an arbitrary number of particles. However they have natural filtrations $C^{\leq k}(-)$ by the subspaces of configurations with at most $k$ points.
    These induce filtrations
        \begin{align*}
            X_k &\coloneqq C^{\leq k}(\Fgb;Z)\parallelsum \diffFgb\\
            Y_k &\coloneqq B\diffFgb \times C^{\leq k}(\mm{R}^\infty;(E\GLt)_+\pointedpushout{\GLt} Z).
        \end{align*}
    that are preserved under the map $\tau\times \delta$. We will inductively show the restrictions $X_k\to Y_k$ are homology isomorphisms for all $k$. 
    
    Since $Z$ is a space with a good basepoint, the inclusions $X_{k-1}\hookrightarrow X_k$ and $Y_{k-1}\hookrightarrow Y_k$ are cofibrations. Their subquotients are
        \[X_k/X_{k-1}=(E\diffFgb)_+\pointedpushout{\diffFgb} \widetilde{C}_k(\Fgb)_+\pointedpushout{\Sigma_k} Z^{\wedge k}\]
    and
        \[Y_k/Y_{k-1}=(B\diffFgb)_+\wedge \widetilde{C}_k(\mm{R}^\infty)_+\pointedpushout{\Sigma_k}((E\GLt)_+\pointedpushout{\GLt} Z)^{\wedge k}.\]
    Comparing the spectral sequences associated to these filtrations, it is enough to show that the induced map on these subquotients is a homology isomorphism. Consider the map of cofibrations:
    
        \begin{adjustbox}{width=\textwidth} 
            $\begin{tikzcd}[column sep=small]
                &&\\
                E\diff(\Fgb)\underset{\diff(\Fgb)}{\times} C_k(\Fgb) \ar[r] \ar[d, "\tau_k\times \varepsilon_k"] &
                E\diff(\Fgb)\underset{\diff(\Fgb)}{\times}     (\widetilde{C}_k(\Fgb)\underset{\,\Sigma_k}{\times} Z^{\wedge k}) \ar[r] \ar[d, "\tau_k\times \varepsilon_k"] & X_k/X_{k-1} \ar[d]\\
                B\diff(\Fgb)\times C_k(\R^\infty,B\GLt) \ar[r]  & B\diff(\Fgb)\times (\widetilde{C}_k(\R^\infty)\underset{\,\Sigma_k}{\times} ((E\GLt)_+\underset{{\GLt}}{\wedge}Z)^{\wedge k}) \ar[r]  & Y_k/Y_{k-1}\\
            \end{tikzcd}$
        \end{adjustbox}
    
    By Lemma \ref{lemma:parametrised-decoupling} with $X=*$ the left-hand map induces a homology isomorphism in degrees $\leq \frac{2}{3}g$, and by the same result with $X=Z^{\wedge n}$, so does the middle map. Then by the five lemma, the right-hand map also induces a homology isomorphism in degrees $\leq \frac{2}{3}g$, as required.
\end{proof}

\subsection{Homological Stability}\label{subsec: homological stability}

Let $F_{g+1,b}$ be a surface of genus $g+1$ and $b\geq 1$ boundary components, which is obtained from $\Fgb$ by a boundary connected sum with $F_{1,1}$. Then extending diffeomorphisms by the identity on $F_{1,1}$ gives a map of topological groups
    \[s:\diff(\Fgb)\hookrightarrow \diff(F_{g+1,b})\]
which we refer to as the \emph{stabilisation map}. Moreover, the inclusion $\Fgo\hookrightarrow F_{g+1,b}$ induces a continuous map of labelled configuration spaces $C(\Fgb;Z)\to C(F_{g+1,b};Z)$ which is $s$-equivariant. Together this implies we have an induced map on the Borel constructions:

\begin{corollary}\label{cor: homological stability for moduli of labelled configurations}
    For $b\geq 1$, the stabilisation map on the Borel constructions
        \[s_*:C(\Fgb;Z)\parallelsum\diffFgb\to C(F_{g+1,b};Z)\parallelsum\diff(F_{g+1,b})\]
    induces a homology isomorphism in degrees $\leq \frac{2}{3}g$.
\end{corollary}

The above result is a corollary of Theorem \ref{thm: decoupling labelled configurations}, however care has to be taken with respect to the model we have used for the Borel constructions and classifying spaces. Namely, it is not clear how to define a stabilisation map on the level of embedding spaces $\Emb(\Fgb,\Rinf)\to \Emb(F_{g+1,b},\Rinf)$ which induces the desired map on Borel constructions. This can be remedied by taking as model for $E\diffFgb$ a weakly contractible subspace of $\Emb(\Fgb,\Rinf)$ that still has a free and has proper action of $\diffFgb$ and in which the stabilisation is clear. 

Fix a boundary component of $\Fgb$ and an embedding $S^1\hookrightarrow\{0\}\times\Rinf$. We denote by $\Emb^\partial (\Fgb,(-\infty,0]\times\Rinf)$ the space of all extensions to an embedding of $\Fgb$ which are standard on a collar neighbourhood of the marked boundary. By the same arguments as above, $\Emb^\partial (\Fgb,(-\infty,0]\times\Rinf)$ is a model for $E\diffFgb$, and the inclusion map 
    \[\Emb^\partial (\Fgb,(-\infty,0]\times\Rinf) \hookrightarrow\Emb(\Fgb,\Rinf)\]
is a $\diffFgb$-equivariant weak homotopy equivalence. Fixing an embedding $e:F_{1,2}\hookrightarrow\Rinf$ which restricts to the chosen embedding on a collar of the boundary, we get an inclusion 
    \[\Emb^\partial (\Fgb,(-\infty,0]\times\Rinf) \to\Emb^\partial (F_{g+1,b},(-\infty,0]\times\Rinf)\]
given by extending any embedding of $\Fgb$ by $e$. This is clearly compatible with the stabilisation map. 

\begin{proof}[Proof of Corollary \ref{cor: homological stability for moduli of labelled configurations}]    
    Using as model for $E\diffFgb$ the space $\Emb^\partial (\Fgb,(-\infty,0]\times\Rinf)$ described above, we get a commutative diagram
    
        \begin{adjustbox}{width=\textwidth} 
        $\begin{tikzcd}[column sep=15mm]
        	{C(\Fgb;Z)\parallelsum \diffFgb} & {C(F_{g+1,b};Z)\parallelsum\diff(F_{g+1,b})} \\
        	{B\diffFgb\times C(\Rinf;(E\GLt)_+\pointedpushout{\GLt} Z)} & {B\diff(F_{g+1,b})\times C(\Rinf;(E\GLt)_+\pointedpushout{\GLt} Z).}
        	\arrow["s", from=1-1, to=1-2]
        	\arrow[from=1-2, to=2-2, "\tau\times\varepsilon"]
        	\arrow[from=1-1, to=2-1, "\tau\times\varepsilon"']
        	\arrow["{s\times id}"', from=2-1, to=2-2]
        \end{tikzcd}$
        \end{adjustbox}
        
    By Theorem \ref{thm: decoupling labelled configurations} the vertical maps induce homology isomorphisms in degrees $\leq\frac{2}{3}g$, and by Harer' Stability Theorem \cite{MR786348,ivaMR896878,ivaMR1015128,ivaMR1234264,boldsen2012improved,RWMR3438379} and Kunneth's Theorem, so does the bottom map. Therefore the top map must also induce homology isomorphisms in degrees $\leq\frac{2}{3}g$.
\end{proof}

The Decoupling Theorem also allows us to identify what the homology stabilises to. Let $C(\Finf;Z)\parallelsum \diffFinf$ and $B\diffFinf$ be respectively
        \begin{align*}
             C(\Finf;Z)\parallelsum \diffFinf\coloneqq & \colim (C(F_{1,1};Z)\parallelsum{\diff(F_{1,1})}\xrightarrow{s} C(F_{2,1};Z)\parallelsum{\diff(F_{2,1})}\xrightarrow{s}\dots)\\
             B\diffFinf\coloneqq & \colim (B\diff(F_{1,1}\xrightarrow{s} B{\diff(F_{2,1})}\xrightarrow{s}\dots).
        \end{align*}
    
\begin{corollary}\label{cor: homology iso on limit spaces - labelled}
    For $Z$ a pointed connected $\GLt$-space, the decoupling map induces 
        \[\tau\times\varepsilon:C(\Finf;Z)\parallelsum \diffFinf \to \Omega^\infty \MTSO(2)\times \Omega^\infty \Sigma^\infty\left(E\GLt_+\pointedpushout{\GLt} Z\right)\]
    which is a homology isomorphism in all degrees.
\end{corollary}

\begin{proof}
    Theorem \ref{thm: decoupling labelled configurations} and Corollary \ref{cor: homological stability for moduli of labelled configurations} imply that the decoupling map on the colimits 
        \begin{align}\label{eq: decoupling colimit}
            \tau\times\varepsilon:C(\Finf;Z)\parallelsum \diffFinf \to B\diffFinf\times C(\Rinf,E\GLt_+\pointedpushout{\GLt} Z)
        \end{align}
    induces a homology isomorphism. By \cite{MR2506750,madsen2007stable}, $B\diffFinf$ admits a map to $\Omega^\infty \MTSO(2)$ which is a homology isomorphism, and by \cite{Segal},  the right-most configuration space in \eqref{eq: decoupling colimit} is homotopy equivalent to $\Omega^\infty \Sigma^\infty\left(E\GLt_+\wedge_{\GLt} Z\right)$.
\end{proof}

\subsection{Monoid of Moduli of Labelled Configuration Spaces}\label{subsec: monoids of moduli of configuration spaces}

Gluing two surfaces $\Fgo$ and $F_{h,1}$ along part of their boundary defines a map of topological groups 
    \[\diff(\Fgo)\times \diff(F_{h,1})\to \diff(F_{g+h,1}).\]
This can be made into an associative operation if we fix once and for all oriented surfaces $\Fgb$ of genus $g$ and one boundary component, and compatible with the stabilisation (see Section \ref{subsec: homological stability}). Using these choices for our surfaces, we can see that the above map gives an associative operation in the collection of diffeomorphism groups of all $\Fgo$. An example of such surfaces and multiplication is depicted in Figure \ref{fig:monoid of configurations}. 

\begin{figure}[bh]
    \centering\def\svgwidth{\columnwidth}
\begingroup%
  \makeatletter%
  \providecommand\color[2][]{%
    \errmessage{(Inkscape) Color is used for the text in Inkscape, but the package 'color.sty' is not loaded}%
    \renewcommand\color[2][]{}%
  }%
  \providecommand\transparent[1]{%
    \errmessage{(Inkscape) Transparency is used (non-zero) for the text in Inkscape, but the package 'transparent.sty' is not loaded}%
    \renewcommand\transparent[1]{}%
  }%
  \providecommand\rotatebox[2]{#2}%
  \newcommand*\fsize{\dimexpr\f@size pt\relax}%
  \newcommand*\lineheight[1]{\fontsize{\fsize}{#1\fsize}\selectfont}%
  \ifx\svgwidth\undefined%
    \setlength{\unitlength}{1466.58332971bp}%
    \ifx\svgscale\undefined%
      \relax%
    \else%
      \setlength{\unitlength}{\unitlength * \real{\svgscale}}%
    \fi%
  \else%
    \setlength{\unitlength}{\svgwidth}%
  \fi%
  \global\let\svgwidth\undefined%
  \global\let\svgscale\undefined%
  \makeatother%
  \begin{picture}(1,0.1271469)%
    \lineheight{1}%
    \setlength\tabcolsep{0pt}%
    \put(0,0){\includegraphics[width=\unitlength,page=1]{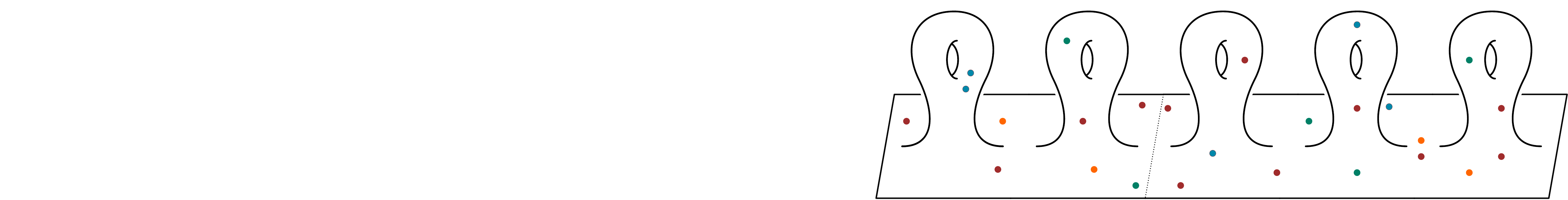}}%
    \put(0.00988108,0.05411934){\makebox(0,0)[rt]{\lineheight{1.25}\smash{\begin{tabular}[t]{r}$\scalebox{1.5}{\ensuremath{\mu}}$\end{tabular}}}}%
    \put(0,0){\includegraphics[width=\unitlength,page=2]{monoidofconfigurations2.pdf}}%
    \put(0.21780384,0.0355019){\makebox(0,0)[t]{\lineheight{1.25}\smash{\begin{tabular}[t]{c}$\scalebox{1.5}{\ensuremath{,}}$\end{tabular}}}}%
    \put(0.51787287,0.05411946){\makebox(0,0)[lt]{\lineheight{1.25}\smash{\begin{tabular}[t]{l}$\scalebox{1.5}{\ensuremath{=}}$\end{tabular}}}}%
    \put(0,0){\includegraphics[width=\unitlength,page=3]{monoidofconfigurations2.pdf}}%
  \end{picture}%
\endgroup%

    \caption{Example of the map $\mu: MC(\mmm{C})_2\times MC(\mmm{C})_3 \to MC(\mmm{C})_{5}$ where $\mmm{C}$ is the space of colours, with white as the basepoint.}\label{fig:monoid of configurations}
\end{figure}

Up to homotopy, the above gluing process is equivalent to gluing the boundary circles of surfaces $\Fgo$ and $F_{h,1}$, to two of the three boundary circles of $F_{0,3}$, what is called the \emph{pair of pants multiplication}. We choose to think of this multiplication in terms of the first description because in this way the product is strictly associative.

This operation also induces a multiplication on the classifying spaces $B\diffFgo$, which can be made associative by picking a convenient model. As in Section \ref{subsec: homological stability}, we will use as $E\diffFgo$ a certain subspace of $\Emb(\Fgo,\Rinf)$. Namely, we can fix embeddings $\delta_g:S^1\hookrightarrow [0,g] \times\Rinf$ such that $\delta_g(S^1)\cap (\{0\}\times\Rinf)=\{0\}\times(-1,1)\times \{0\}$ and $\delta_g(S^1)\cap (\{g\}\times\Rinf)=\{g\}\times(-1,1)\times \{0\}$. We denote by $\Emb^{\delta_g} (\Fgo,[0,g]\times\Rinf)$ the space of all extensions to an embedding of $\Fgo$ which are standard on a collar neighbourhood of the boundary. By the same arguments as above, $\Emb^{\delta_g} (\Fgo,[0,g]\times\Rinf)$ is a model for $E\diffFgo$, and the inclusion map 
    \[\Emb^{\delta_g} (\Fgo,[0,g]\times\Rinf) \hookrightarrow\Emb(\Fgo,\Rinf)\]
is a $\diffFgo$-equivariant weak homotopy equivalence. By picking embeddings $\delta_g$ which are compatible with the stabilisation map, we can define a multiplication
    \[\Emb^{\delta_g} (\Fgo,[0,g]\times\Rinf)\times \Emb^{\delta_h} (F_{h,1},[0,h]\times\Rinf)\to\Emb^{\delta_{g+h}} (F_{g+h,1},[0,g+h]\times\Rinf)\]
by extending an embedding of $\Fgo$ by the translation of the embedding of $F_{h,1}$ in the first coordinate by $g$. This is clearly associative and it is compatible with the associative multiplication on the diffeomorphism groups.

Taking as model for $E\diffFgo$ the spaces  $\Emb^{\delta_g} (\Fgo,[0,g]\times\Rinf)$ we obtain an associative multiplication on classifying spaces. This structure equips the space $\dcup{g\geq 0}B\diffFgo$ with the structure of a topological monoid, which we refer to as the \emph{surface monoid}. This is equivalent to the one studied by Tillmann in \cite{Tillmann_surface_operad}, which was essential in the proof of the Madsen-Weiss Theorem \cite{madsen2007stable}.

The operation on diffeomorphism groups and spaces $E\diffFgo$ described above, together with the fixed identifications $F_{g+h,1}=\Fgo\#_\partial F_{h,1}$, induce an associative multiplication also on the Borel constructions (see Figure \ref{fig:monoid of configurations})
    \[\mu:C(\Fgo;Z)\parallelsum\diffFgo\times C(F_{h,1};Z)\parallelsum\diff(F_{h,1})\to C(F_{g+h,1};Z)\parallelsum\diff(F_{g+h,1}).\]

\begin{definition}    
We denote by $MC(Z)_g\coloneqq C(\Fgo;Z)\parallelsum\diffFgo$. The \emph{monoid of moduli of configurations labelled by $Z$} is the topological monoid given by the disjoint union $MC(Z)\coloneqq\dcup{g\geq 0}MC(Z)_g$ together with the operation $\mu$.
\end{definition}    

\begin{theorem}\label{thm:group-completion}
    For any pointed $\GLt^+$-space $Z$, the decoupling map induces a weak equivalences on group completions
        \[\Omega B\left(\dcup{g}C(\Fgo;Z)\parallelsum\diffFgo\right) \simeq \Omega B\left(\dcup{g}B\diffFgo\right)\times \Omega B\,C(\Rinf;E\GLt_+\pointedpushout{\GLt}Z).\]
\end{theorem}

Segal showed in \cite{Segal} that the space $\Omega^\infty\Sigma^\infty(X)$ is the group completion of the configuration space $C(\mm{R}^\infty; X)$ seen as a topological monoid where the operation is roughly given by transposition. Using the inclusion $\Emb^{\delta_g} (\Fgo,[0,g]\times\Rinf)\hookrightarrow\Emb(\Fgb,\Rinf)$, we see that the decoupling map of Theorem \ref{thm: decoupling labelled configurations} induces a map between the spaces using the current model. The proof of the above theorem will consist of showing that the decoupling induces a monoidal map between $MC(Z)$ and the monoids Tillmann and Segal, and to show that this induces a homotopy equivalence on group completions.
        
\begin{lemma}\label{lemma:group completion argument}
    The maps $\tau$ and ${\varepsilon}$ defined in \ref{eq: forgetful and evaluation maps} are compatible with these monoidal structures.
\end{lemma}

For the map $\tau$, this result follows directly from the definition. For $\varepsilon$, some care has to be taken into making the configuration spaces into actual topological monoids (see \cite{Segal}). Namely, instead of $C(\Rinf; X)$, we use the homotopy equivalent space
    $${C}'(\Rinf,X)=\{(c,t)\in {C}(\Rinf,X)\times \mm{R}:t\geq 0, c\subset (0,t)\times \Rinf\}.$$ 
The monoidal structure is given by \emph{juxtaposition}, ie. $(c,t),(c',t') \mapsto (c\cup T_t (c'),t+t')$
where $T_t(-)$ is the map that translates a configuration by $t$ on the first direction. The map $\tau$ of the decoupling theorem is then equivalent to the monoidal map $MC(Z)\to {C}'(\Rinf,X)$ taking an element of $[e,c]\in MC(Z)_g$ to $(\tau([e,c]),g)$.

\begin{proof}[Proof of Theorem \ref{thm:group-completion}]
It is enough to show that the map of monoids given by the decoupling induces a homotopy equivalence on group completions. As the group completions are loop spaces, they are in particular simple and, by the Whitehead theorem for simple spaces, it suffices to show that it induces a homology equivalence on the group completions. Both monoids are homotopy commutative, hence the group completion theorem \cite{McDuff-Segal} can be applied. Therefore it is enough to prove that the induced map on the limit spaces (defined in Section \ref{subsec: homological stability})
    \begin{align}\label{eq:limit-spaces-for-group-completion}
        \tau\times\varepsilon:C(\Finf;Z)\parallelsum \diffFinf \to B\diffFinf \times C(\mm{R}^\infty;E\GLt_+\pointedpushout{\GLt} Z)
    \end{align}
is a homology equivalence. This holds by Theorem \ref{thm: decoupling labelled configurations} and Corollary \ref{cor: homological stability for moduli of labelled configurations}.
\end{proof}

\section{Decoupling Configuration Spaces with Partially Summable Labels}\label{chap: configuration spaces summable labels}

In this section we prove the main result of this paper, which is a decoupling result for configuration spaces with partially summable labels. In this case, the labelling space is equipped with a partial multiplication and the particles are allowed to collide whenever their labels can be multiplied. The space $\Csum(M;P)$ of configurations in $M$ with partially summable labels in $P$ has been defined in \cite{Salvatore1999Configuration} and its definition requires more sophisticated tools such as the Fulton-MacPherson configuration spaces and operad. In section \ref{sec:partial d-monoids} we recall these definitions and the concept of a partial $d$-monoids. 

To prove the decoupling theorem for $\Csum(M;P)$, we develop a semi-simplicial resolution for this space in section \ref{sec:poset-surrounded-config}, denoted $|\Tsum(M;P)_\bullet|$. In Proposition \ref{prop: weak equivalence with poset one} we show that indeed this space is weakly equivalent to $\Csum(M;P)$. In the decoupling context, we naturally encounter another space of discs with configurations which is constructed in Definition \ref{def:thin-discs}.

With these semi-simplicial spaces, we prove Theorem \ref{main thm: decoupling summable labels}, combining Corollary \ref{cor:disc model for borel construction} and Theorem \ref{thm: decoupling summable configurations}. We then use this to deduce Theorem \ref{mainthm: homological stability for moduli of summable configurations} and Corollary \ref{mainthm: group completion summable}, which are, respectively Corollaries \ref{cor: homological stability for moduli of summable configurations} and \ref{thm: group completion summable}.

\subsection{Fulton-MacPherson configuration space and partially summable labels}\label{sec:partial d-monoids}

In this section we recall the concept of a configuration space with summable labels in a partial monoid as described in \cite{Salvatore1999Configuration}.

A partial $d$-monoid $P$ is, in essence, a space with a continuous operation that is not defined on all collections of elements, but only on a subset of \emph{composable elements}. The data defining such a partial monoid consists of the subset of composable elements, together with an operation on this subset which is associative. We are interested in partial monoids that moreover have the structure of an $E_d$-algebra. 
Essential to this construction is $\F_d$, the Fulton-MacPherson operad in dimension $d$. This is a cofibrant replacement for the little $d$-discs operad and therefore will be used to make precise the notion of a $E_d$-partial monoid. Crucially for us, the Fulton-MacPherson operad is defined in terms of configuration spaces of points, in which the particles in the configuration are allowed to collide, but keeping track of the relative position of the particles before they collided and the order in which collided. 

We follow the definition of the Fulton-MacPherson configuration space $\overline{C}_k(M)$ of a manifold $M$ as described in \cite{Sinha04}. To make it precise, we consider $M$ as embedded in some $\R^m$ and we denote by $\underline{k}$ the set $\{1,\dots,k\}$. We record the directions of particles in a collision, by defining for $(i,j)\in \widetilde{C}_2(\underline{k})$, a map $\pi_{i,j}:\widetilde{C}_k(\R^m)\to S^{m-1}$ sending a configuration $(x_1,\dots,x_k)$ to the unit vector in the direction $x_i-x_j$. The order of collision is recorded by defining for $(i,j,\ell)\in \widetilde{C}_3(\underline{k})$ the map $s_{i,j,\ell}:\widetilde{C}_k(\R^m)\to [0,\infty]$ sending $(x_1,\dots,x_k)$ to $|x_i-x_j|/|x_i-x_\ell|$.

\begin{definition}[\cite{Sinha04}, Definition 1.3]\label{def: fulton mac config space}
     The \emph{Fulton-MacPherson configuration space} $\overline{C}_k(M)$ of the manifold $M$ is the closure of the image of the map
        \[i\times (\pi_{i,j}|_{\widetilde{C}_k(M)})\times (s_{i,j,k}|_{\widetilde{C}_k(M)}):\widetilde{C}_k(M)\longrightarrow M^k\times (S^{m-1})^{k(k-1)}\times [0,\infty]^{k(k-1)(k-2)}.\]
\end{definition}

The space $\overline{C}_k(M)$ is homotopy equivalent to the ordered configuration space $\widetilde{C}_k(M)$ \cite[Corollary 4.5]{Sinha04} and whenever $M$ is compact, $\overline{C}_k(M)$ is a compactification of $\widetilde{C}_k(M)$. Moreover, this construction is functorial with respect to embeddings \cite[Corollary 4.8]{Sinha04}, i.e. any embedding $f:M\hookrightarrow N$ induces an embedding $f_*:\fmconf_k(M)\to \fmconf_k(N)$.

The following result gives a convenient way to represent elements of $\fmconf_k(M)$, which will be used throughout the chapter.

\begin{proposition}[\cite{Sinha04}, Theorem 3.8]\label{prop:describing the fulton macpherson}
  Each element in $\overline{C}_k(M)$ is uniquely determined by:  
    \begin{enumerate}
        \item A configuration of points $P_1,\dots, P_l$ in the interior of $M$, with $1\leq l\leq n$ (we refer to these as the \emph{infinitesimal configuration}),
        \item For each $1\leq i\leq l$, a tree $T_i$ with $f_i$ leaves (twigs), no bivalent vertices, so that $\sum_{i=1}^l f_i=n$, and for each vertex in $T_i$ of valence $m$ an element in $C_m(T_{P_i}M)/G(d)$, where $G(d)$ is the group of affine transformations of $\R^d$ generated by translations and positive dilations.
        \item A global ordering of the $k$ leaves of the trees.
    \end{enumerate}
\end{proposition}

This interpretation of the elements of $\fmconf_k(M)$ also provides a way of describing the map $f_*:\fmconf_k(M)\to \fmconf_k(N)$ induced by an embedding $f:M\hookrightarrow N$ into a $d'$-manifold $N$: it takes a point with infinitesimal configurations $P_1,\dots,P_l\in M$ to one with infinitesimal configurations $f(P_1),\dots,f(P_l)\in N$, it preserves the trees and ordering of the leaves, but changes the labels of the vertices of the trees by taking a label $\xi\in C_m(T_{P_i}M)/G(d)$ to the label $D_{P_i}f(\xi)\in C_m(T_{f(P_i)}N)/G(d')$.

The Fulton-MacPherson operad $\F_d$ is built out of subspaces of these configurations in $\R^d$. Intuitively, for each $k\geq 0$ the space $F_d(k)$ is the subspace of the ordered configurations of $k$ points in $\R^d$, in which the points have collided at the origin. 

\begin{definition}\label{defn: fulton mac operad}
    The \emph{Fulton-MacPherson operad} in dimension $d$, denoted $\F_d$, is defined by taking $\F_d(k)$ to be the pullback
      \[\begin{tikzcd}
            \F_d(k) \ar[r] \ar[d]  \ar[dr, phantom, "\lrcorner", very near start] & \overline{C}_k(\R^d) \ar[d]\\
            0 \ar[r] & (\R^d)^k.
        \end{tikzcd}\]
    In other words, it is the subspace of $\overline{C}_k(\R^d)$ with infinitesimal configuration given by a single point at the origin. 
    
    With the description of this space as in Proposition \ref{prop:describing the fulton macpherson}, the composition of this operad is given by grafting of trees. Pictorially, we will often represent elements of $\F_d(k)$ as trees of configurations such as in the rightmost picture of Figure \ref{fig:element of the FM operad}. 
\end{definition}

\begin{figure}[h!t]
    \centering
    \def\svgwidth{\textwidth}
    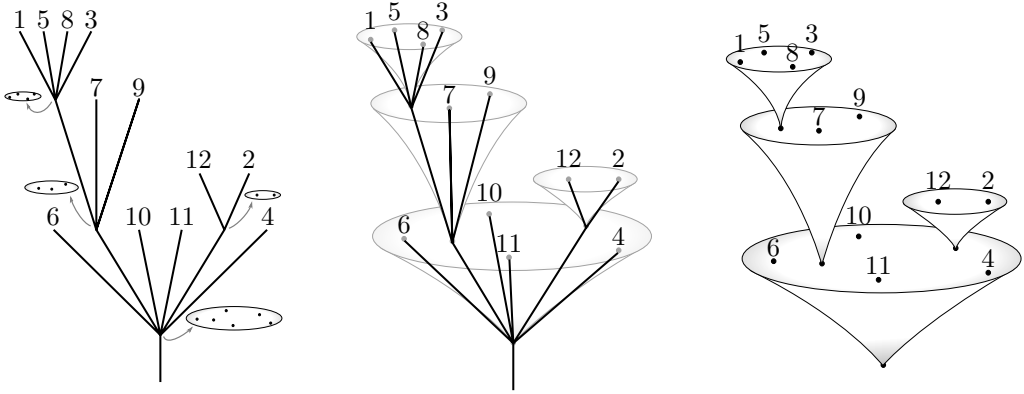
    \caption{Representation of an element in the space $\F_2(12)$ in terms of a tree of infinitesimal configurations.}
    \label{fig:element of the FM operad}
\end{figure}

As shown in \cite{Salvatore1999Configuration}, there exists a model structure on the category of topological operads in which $\F_d$ is a cofibrant replacement of the little $d$-discs operad. An algebra over $\F_d$, called by Salvatore a \emph{$d$-monoid}, consists of a space $A$ together with $\Sigma_k$-equivariant maps
    \[\F_d(k)\times A^k\to A\]
which commute with the structure maps of $\F_d$.

\begin{definition}[\cite{Salvatore1999Configuration}, Definition 1.7 and 2.6]\label{def:partial d-monoid}
    A \emph{partial $d$-monoid} is a space $P$ 
    together with monomorphisms of $\Sigma_k$-spaces
        \[i:Comp_k\hookrightarrow \F_d(k)\times_{\Sigma_k}P^k.\]
    and composition maps $\rho:Comp_k\to P$ such that 
    \begin{enumerate}
        \item The unit $\eta\times P:P\to \F_d(1)\times P$ factors uniquely through $\widetilde{\eta}:P\to Comp_1$ and the composition $\rho\circ\widetilde{\eta}$ is the identity $\id_P$;
        \item For $f\in \F_d(k)$ and $[f_j;\overline{p}_j]\subset Comp_{m_j}$, for $j=1,\dots,k$, the element
            \[[f; \rho(f_1;\overline{p}_1),\dots, \rho(f_k;\overline{p}_k)]\in \F_d(k)\times_{\Sigma_k}P^k\]
        belongs to $Comp_k$ if and only if 
            \[[\mu(f;f_1,\dots,f_k);\overline{p}_1,\dots,\overline{p}_k]\]
        belongs to $Comp_{m_1+\dots+m_k}$. Moreover, if that is the case, then their image by $\rho$ coincide.
    \end{enumerate}
    A pointed space $(P,0)$ is a \emph{partial $d$-monoid with unit} \index{partial $d$-monoid! with unit} if in addition it satisfies
        \begin{enumerate}[3.]
            \item For every $k$ there is an inclusion $u:\F_d(k)\times_{\Sigma_k} \bigvee_k P\hookrightarrow Comp_k$ such that the composition with $i\circ u$ is the subspace inclusion, and $\rho\circ i:\F_d(k)\times_{\Sigma_k} \bigvee_k P\to P$ is the projection onto the $P$ coordinate.
        \end{enumerate}
\end{definition}

The definition above is better understood when in comparison to $E_d$-algebras, which Salvatore calls $d$-monoids. Given a $d$-monoid $M$ and a infinitesimal configuration $f\in \F_d(k)$, we can always compose any $k$ elements of $M$ via the composition rule described by $f$. In a partial $d$-monoid $P$, that is not the case. We instead are given a subset of the $k$-tuples of $P$ which are composable via the operation described by a given $f\in \F_d(k)$. The space $Comp_k$ should be then thought of as the pairs of possible composition rules in $\F_d$ together with the tuples of elements of $P$ which can be composed with this composition rule. The map $\rho$ is computes the results of compositions, whenever they are defined.

\begin{example}\label{ex: partial monoids}
\begin{enumerate}[(a)]
    \item Any space $X$ admits the structure of a trivial partial $d$-monoid, by defining $Comp_1=\{[1,x]:x\in X\}$ and $Comp_k=\emptyset$, for all $k\neq 1$. In this case, the only composition rule that can be performed is the identity and no other collection of points in $X$ is composable in any way.
    \item\label{ex:trivial partial monoid} When $X$ is a space equipped with a basepoint $*$, we can define a unital partial $d$-monoid by setting $Comp_k=\F_d(k)\times\vee_k X$ and $\rho(f,\overline{x})=x_i$, where $x_i$ is the unique non-basepoint coordinate, or $*$ otherwise. In this case, the basepoint acts as a unit and compositions are only defined when done with the unit.
    \item Every $d$-monoid is trivially a partial $d$-monoid. In particular, every $\Omega^d$-space is a partial $d$-monoid.
    \item\label{item: pm - naive upgrade} The canonical inclusion $i_{d,n}:\R^d\hookrightarrow \R^{d+n}$ also allows us to construct a partial $(d+n)$-monoid $P$ from a partial $d$-monoid by adding $n$ trivial composition directions. We call this the \emph{naive upgrade of a partial $d$-monoid} $P$ and denote it $T_{d,n}P$. The underlying space of $T_{d,n}P$ is $P$, and we define $Comp_k^{T_{d,n}P}$ to be the image of
        \[Comp_k^P\hookrightarrow \F_d(k)\times_{\Sigma_k}P^k \xhookrightarrow{(i_{n,k})_*}\F_{d+n}(k)\times_{\Sigma_k}P^k=\F_{d+n}(k)\times_{\Sigma_k}(T_{d,n}P)^k.\]
\end{enumerate}
\end{example}

In \cite{Salvatore1999Configuration}, it was always assumed that the partial $d$-monoids were \emph{good}, in the sense described below. 

\begin{definition}
    A partial $d$-monoid $P$ is \emph{good}\index{partial $d$-monoid!good} if for every $k$ the inclusion $Comp^P_k\hookrightarrow \F_d(k)\times_{\Sigma_k}P^k$ is a cofibration. 
\end{definition}

From now on, we always assume partial $d$-monoids to be good and to have a unit. 
For future applications, we are further interested in partial $d$-monoids with compatible actions of $\GLd$, so we introduce this concept here and perform our constructions in this more general setting. 

Recall from Proposition \ref{prop:describing the fulton macpherson} that an element of $\F_d(k)$ is described by a tree with $k$ ordered leaves and a decoration of the vertices by elements of $\widetilde{C}_{|v|}(\R^d)/G(d)$. The action of $\GLd$ on $\R^d$ induces an action on $\widetilde{C}_{|v|}(\R^d)/G(d)$. This gives an action of this group on $\F_d(k)$ for every $k$, and it is simple to check that the operad maps $\mu$ are all $\GLd$ equivariant.

\begin{definition}[\cite{Salvatore1999Configuration}, Definition 4.3]
    The \emph{framed Fulton-MacPherson operad}\index{Fulton-MacPherson operad!framed} denoted $f\F_d$, is the operad defined by $f\F_d(k)=\F_d(k)\times \GLd^k$, with structure map 
        \begin{align*}
            \widetilde{\mu}((x,g_1,\dots,g_k);&(x_1,g_1^1,\dots,g_1^{m_1}),\dots, (x_k,g_k^1,\dots,g_k^{m_k}))\\
            & =(\mu(x;g_1x_1,\dots,g_kx_k),g_1g_1^1,\dots,g_kg_k^{m_k}).
        \end{align*}
\end{definition}

The construction above is an instance of the construction $\mmm{A}\rtimes G$, the semi-direct product of an operad $\mmm{A}$ and group $G$. This construction, and the proof that the above indeed defines an operad can be found in \cite[Definition 2.1]{SalvatoreWahl}.

\begin{definition}[\cite{Salvatore1999Configuration}, Definition 4.8]
    A \emph{framed partial $d$-monoid}\index{partial $d$-monoid!framed} with unit is a pointed $\GLd$-space $P$ together with monomorphisms of $\Sigma_k\wr\GLd$-spaces
        \[i:fComp_k\hookrightarrow f\F_d(k)\times_{\Sigma_k}P^k.\]
    and $\GLd$-equivariant composition maps $\rho:fComp_k\to P$ satisfying properties 1-3 of Definition \ref{def:partial d-monoid}.
\end{definition}

The $\GLd$-bundle of frames on $M$ induces a $(\GLd)^k$-bundle $f\fmconf_k(M)$ on $C_k(M)$, acted on by $\Sigma_k$. This is called the \emph{framed configuration space}\index{configuration space!framed} of $k$ points in $M$. Using the description of Proposition \ref{prop:describing the fulton macpherson}, an element of the space $f\fmconf_k(M)$ can be uniquely determined by an infinitesimal configuration in $M$ with labelled trees, together with additional $k$ frames of the tangent planes associated to the $k$ leaves of the trees. 

Then the space of framed configurations $f\fmconf(M)=\dcup{k}f\fmconf_k(M)$ is a right $f\F_d$-module, with $\GLd$-equivariant multiplication maps 
    \[m:f\fmconf_k(M)\times_{\Sigma_k} (f\F_d(n_1) \times \dots \times f\F_d(n_k))\to f\fmconf_{n_1+\dots+n_k}(M).\]
defined by grafting the element of $f\F_d(n_i)$ on the $i$-th leaf of the element of $f\fmconf_k(M)$, for all $i=1,\dots,k$ and using the frame on the leaves to identify $C_m(\R^d)/G(d)$ with a configuration on the tangent space of $M$ (\cite[Proposition 4.5]{Salvatore1999Configuration}). It is simple to verify that any co-dimension zero embedding $e:M\hookrightarrow N$ induces a right $f\F_d$-homomorphism $e_*:f\fmconf(M)\hookrightarrow f\fmconf(N)$.

\begin{definition}[\cite{Salvatore1999Configuration}, Definition 4.14]\label{def:configs with summable labels}
    Let $P$ be a framed partial $d$-monoid, and let $M$ be a manifold of dimension $d$. Then the space of configurations in $M$ with partially summable labels in $P$, denoted $\Csum(M;P)$ is defined as the co-equalizer of the following
        \[\begin{tikzcd}[column sep=65pt]
            {\coprod\limits_{k}\left[ f\fmconf_k(M)\times_{\Sigma_k} \left( \coprod\limits_{\pi\in\Map(n,k)} \prod\limits_{i=1}^k fComp_{\pi^{-1}(i)}^P
            \right)\right]} \ar[r, shift right, "{(m\times \id)\circ(\id\times i)}"'] \ar[r, shift left, "\id\times \rho^k"] & {\coprod\limits_{k} f\fmconf_k(M)\times_{\Sigma_k} P^k} 
        \end{tikzcd}\]
\end{definition}

An element of $\Csum(M;P)$ is then an equivalence class of elements in $\dcup{k} f\fmconf_k(M)\times_{\Sigma_k} P^k$.  
From the description of $f\fmconf_k(M)$ above, we can see that an element in $\Csum(M;P)$ can be represented by an infinitesimal configuration $w_1,\dots,w_l$ of $\ell<k$ points in $M$ together with trees $T_i$, for $i=1,\dots,\ell$, where the vertices of $T_i$ are labelled by elements in $x^i_j\in f\F_d$, and the leaves of the trees are labelled by elements of $p_k^i\in P$. The equivalence relation defining $\Csum(M;P)$ implies that if some leaves labelled by $p_1,\dots,p_k$ are departing from a vertex labelled by $x\in f\F_d(k)$ and the composition $\rho(x;p_1,\dots,p_k)$ is defined, then we identify this configuration with the one in which such leaves are removed and their vertex is replaces by a leaf labelled by $\rho(x;p_1,\dots,p_k)$.

For a framed partial $d$-monoid $P$ and a co-dimension zero embedding $f:M\hookrightarrow N$, we get an induced map
    \[f_*:\coprod\limits_{k} f\fmconf_k(M)\times_{\Sigma_k} P^k\to \coprod\limits_{k} f\fmconf_k(N)\times_{\Sigma_k} P^k.\]
With the description above, a point in the domain with infinitesimal configurations $w_1,\dots,w_\ell$ in $M$, with trees $T_i$, for $i=1,\dots,\ell$, where the vertices of $T_i$ are labelled by elements $x_j^i\in f\F_d$, and the leaves of the trees are labelled by elements of $p_k^i\in P$, is taken to the point with infinitesimal configuration $\phi(x_1),\dots,\phi(x_\ell)$, trees $T_i$, $i=1,\dots,\ell$, and corresponding labels $$d_{x_i}\phi\cdot x^i_j\, \text{ and }\, d_{x_i}\phi\cdot p_k^i.$$ Here we are using the standard actions of $\GLd$ on $f\F_d$ and $P$.
By the equivariance condition in the definition of a framed partial monoid, the map preserves the equivalence classes described above. Therefore any such co-dimension zero embedding induces a map 
    \begin{align}\label{eq: map induced by cod-0 embedding}
        f_*:\Csum(M;P)\to \Csum(N;P).
    \end{align}

Seeing the group $\diff(M)$ as a subspace of $\Emb(M,M)$, the above construction shows that $\Csum(M;P)$ admits an action of $\diff(M)$. We will be interested in a decoupling theorem for the space $\Csum(M;P)\parallelsum\diff(M)$.

\subsection{Disc models for configuration spaces with partially summable labels}\label{sec:poset-surrounded-config}

Let $M$ be a smooth compact $d$-manifold, possibly with boundary, and $P$ a framed partial $d$-monoid. The group $\Sigma_k\wr\GLd$ has a canonical inclusion into $\diff(\dcup{k}\R^d)$. We consider $\Emb(\dcup{k} \R^d, M)$ to be a $\Sigma_k\wr\GLd$-space with action given by pre-composition by the inverse, and $\Csum(\dcup{k}\R^d;P)$ to be a $\Sigma_k\wr\GLd$-space with action induced by the action of $\diff(\dcup{k}\R^d)$, as described in the previous section.

\begin{definition}[\cite{Manthorpe-Tillmann}]\label{def: tubular configurations}
    Let $Z$ be a pointed $\GLd$-space, The \emph{space of tubular configurations} in $M$ with labels in $Z$, denoted $\Tub(M;Z)$, is the quotient
        \[\left(\coprod\limits_{k\geq 0}  \Emb(\dcup{k} \R^d, M) \pushout{\Sigma_k\wr \GLd} Z^k \right)\bigg/\sim\]
    where $(e_1,\dots,e_k;z_1,\dots,z_k)\sim(e_1,\dots,e_{k-1};z_1,\dots,z_{k-1})$
    whenever $z_k$ is the basepoint of $Z$, for $e_i:\R^d\hookrightarrow M$ and $z_i\in Z$. 
\end{definition}

The space $\Tub(M;Z)$ is equipped with an action of $\diff(M)$: for $\psi\in\diffM$, an embedding $e:\dcup{k}\R^d\hookrightarrow M$, and $\overline{z}=(z_1,\dots,z_k)\in Z^k$, we define
    \[\phi\cdot [e,\overline{z}]=[\phi\circ e;D_{e(0_1)}\phi\cdot z_1,\dots,D_{e(0_k)}\phi\cdot z_k]\]
where $0_i$ denotes the origin of the $i$-th component of $\dcup{k}\R^d$.

\begin{lemma}[\cite{Manthorpe-Tillmann}, Propositions 2.7, 2.8, and 3.6]\label{lemma: tubular configurations}
    The inclusion of the origin $i:*\hookrightarrow\R^d$ induces a $\diffM$-equivariant weak equivalence
        \[i^*:\Tub(M;Z)\xrightarrow{\simeq} C(M;Z).\]
\end{lemma}

One can think of $\Tub(M;Z)$ as disc models for configuration spaces with labels. To construct a disc model for summable labels, we need much more structure:

\begin{definition}\label{def: space of surrounded configurations}
    The \emph{space of surrounded configurations} in $M$ with summable labels in $P$, denoted $\Dsum(M;P)$, is the quotient
        \[\left(\coprod\limits_{k\geq 0}  \Emb(\dcup{k} \R^d, M) \pushout{\Sigma_k\wr \GLd} \Csum(\dcup{k}\R^d;P) \right)\bigg/\sim\]
    where $(e:\dcup{m}\R^d\to M,\xi)\sim(e':\dcup{n}\R^d\to M,\xi')$ if there are injections $\underline{k}\hookrightarrow\underline{m}$ and $\underline{k}\hookrightarrow\underline{n}$ such that the induced inclusions $i_1:\dcup{k}\R^d\to \dcup{m}\R^d$ and $i_2:\dcup{k}\R^d\to \dcup{n}\R^d$ satisfy
        \begin{align*}
            \xi\subset Im\, i_1, \,& \xi'\subset Im\, i_2 & e\circ i_1&=e'\circ i_2, & \text{and } e_*(\xi)=e'_*(\xi').
        \end{align*}
    Here $e_*$ denotes the map of configurations with summable labels induced by a codimension zero embedding, as described in \eqref{eq: map induced by cod-0 embedding}. Since the Fulton-MacPherson configuration spaces are functorial for co-dimension zero embeddings, we get a map
        \begin{align}\label{eq: augmentation map on 0}
            p:\Dsum(M;P)\to \Csum(M;P)
        \end{align}
    taking a class $(e,\xi)$ to the configuration $e_*(\xi)$. By definition, this map does not depend on the choice of representative $(e,\xi)$. 
    The space $\Dsum(M;P)$ admits a partial ordering by declaring  $(e,\xi)<(e',\xi')$ if, for some representative of the classes, $e(\xi)=e'(\xi')$ and $Im\, e\supset Im\, e'$. We denote by $\Dsum(M;P)_\bullet$ the semi-simplicial nerve of the poset $\Dsum(M;P)$. 
\end{definition}

By the definition of the partial order, the map $p$ induces an augmentation $\Dsum(M;P)_\bullet\to \Csum(M;P)$. In particular, $\Dsum(M;P)_\bullet$ is an augmented topological flag complex (Definition \ref{def: topological flag complex}), since the space of $n$-simplices is indeed an open subspace of the $(n+1)$-tuples of vertices with the same image under the augmentation map, and the condition for a tuple to form an $n$-simplex is just given by checking the pairwise order relation. 
It might be helpful to keep in mind the following visualisation for elements in the space $|\Dsum(M;P)_\bullet|$: any point can be expressed by a tuple 
    \begin{align*}
        ((e_0,\xi_0)<\dots<(e_k,\xi_k);t_0,\dots,t_k) && t_0\dots+t_k=1.
    \end{align*}
The poset construction implies that for such a tuple $e_0(\xi_0)=\dots=e_k(\xi_k)$ and $Im\, e_0\supset\dots \supset Im\, e_k$. We can visualise it as a collection of descending embedded discs around the configuration $e_i(\xi_i)$ in $M$. Each collection of embedded discs has weights adding up to $1$ and when the weight associated to a collection of embeddings goes to zero, that collection disappears. The equivalence relation guarantees that we can always get a representative for the collection of discs such that each component has at least one particle of the configuration inside it. See Figure \ref{fig:visualising the poset}.

\begin{figure}[h!t]
    \centering
    \def\svgwidth{0.5\textwidth}
\begingroup%
  \makeatletter%
  \providecommand\color[2][]{%
    \errmessage{(Inkscape) Color is used for the text in Inkscape, but the package 'color.sty' is not loaded}%
    \renewcommand\color[2][]{}%
  }%
  \providecommand\transparent[1]{%
    \errmessage{(Inkscape) Transparency is used (non-zero) for the text in Inkscape, but the package 'transparent.sty' is not loaded}%
    \renewcommand\transparent[1]{}%
  }%
  \providecommand\rotatebox[2]{#2}%
  \newcommand*\fsize{\dimexpr\f@size pt\relax}%
  \newcommand*\lineheight[1]{\fontsize{\fsize}{#1\fsize}\selectfont}%
  \ifx\svgwidth\undefined%
    \setlength{\unitlength}{409.14758135bp}%
    \ifx\svgscale\undefined%
      \relax%
    \else%
      \setlength{\unitlength}{\unitlength * \real{\svgscale}}%
    \fi%
  \else%
    \setlength{\unitlength}{\svgwidth}%
  \fi%
  \global\let\svgwidth\undefined%
  \global\let\svgscale\undefined%
  \makeatother%
  \begin{picture}(1,0.2771984)%
    \lineheight{1}%
    \setlength\tabcolsep{0pt}%
    \put(0,0){\includegraphics[width=\unitlength,page=1]{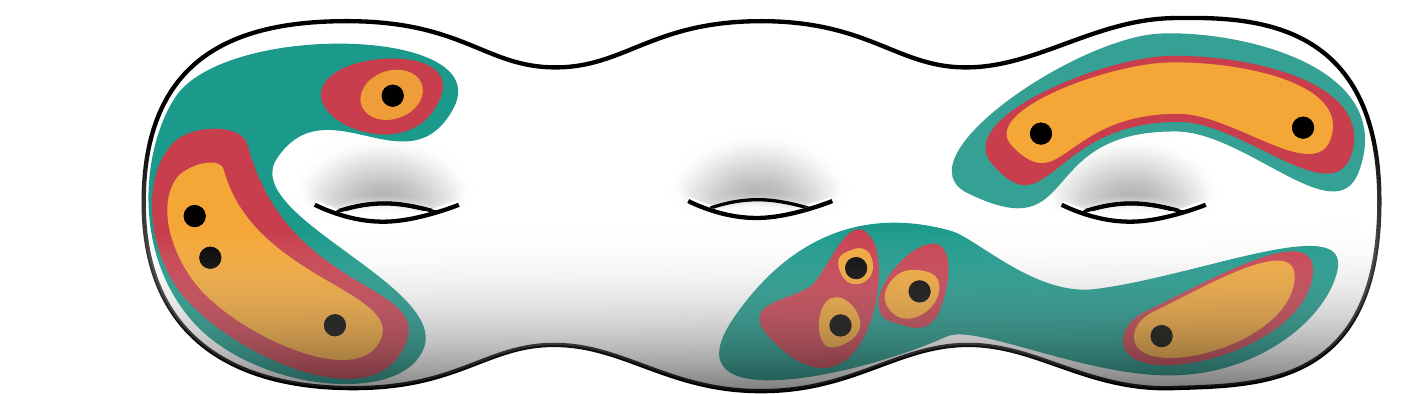}}%
    \put(0.13683004,0.24869644){\color[rgb]{0.10588235,0.6,0.54509804}\makebox(0,0)[rt]{\lineheight{1.25}\smash{\begin{tabular}[t]{r}$\bm{t_0}$\end{tabular}}}}%
    \put(0.09894836,0.06387053){\color[rgb]{0.89411765,0.65882353,0.31372549}\makebox(0,0)[rt]{\lineheight{1.25}\smash{\begin{tabular}[t]{r}$\bm{t_2}$\end{tabular}}}}%
    \put(0.09100541,0.15093729){\color[rgb]{0.78431373,0.24313725,0.30196078}\makebox(0,0)[rt]{\lineheight{1.25}\smash{\begin{tabular}[t]{r}$\bm{t_1}$\end{tabular}}}}%
  \end{picture}%
\endgroup%

    \caption{Element in a $2$-simplex of $|\Dsum(F_3)|$.}
    \label{fig:visualising the poset}
\end{figure}
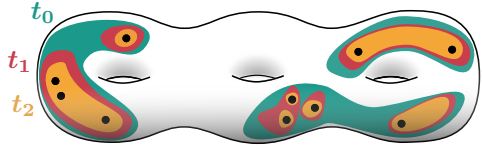

\begin{proposition}\label{prop: weak equivalence with poset one}
    The map $p:\Dsum(M;P)\to \Csum(M;P)$ induces a weak homotopy equivalence
        \[|\Dsum(M;P)_\bullet|\to \Csum(M;P).\]
\end{proposition}

The proof of the above will be a direct application of Theorem \ref{thm:augmented semi-simplicial spaces and realisation} \cite[Theorem 6.2]{MR3207759}.

\begin{proof}
    As discussed before, the augmented semi-simplicial space $p:\Dsum(M;P)_\bullet\to \Csum(M;P)$ is an augmented topological flag complex. We need to verify that the hypotheses of Theorem \ref{thm:augmented semi-simplicial spaces and realisation} are satisfied:
        \begin{enumerate}[(i)]
            \item To see that $p:\Dsum(M;P)\to \Csum(M;P)$ has local sections, take $f:D^n\to \Csum(M;P)$ and a point $(e,\xi)\in p^{-1}(f(x))$. Then the image of $e$ is an open subset of $M$ containing $f(x)$, and therefore the subspace $V$ of configurations contained in $\Ima e$ is an open subset of $\Csum(M;P)$ containing $f(x)$. Let $U=f^{-1}(V)$, which is an open neighbourhood of $x$ in $D^n$. Then the map $F:U\to \Dsum(M;P)$ taking $y$ to $(e,e^{-1}(f(y)))$ satisfies $p\circ F=f|_U$ and $F(x)=(e,\xi)$.
            \item $p:\Dsum(M;P)\to \Csum(M;P)$ is surjective, since any configuration $\xi$ admits a tubular neighbourhood $e$. Then $(e,e^{-1}(\xi))$ is an element of $\Dsum(M;P)$ in the pre-image of $\xi$.
            \item For a configuration $\xi\in \Csum(M;P)$ and a non-empty finite subset $\{(e_1,\xi_1),\dots,(e_k,\xi_k)\}$ in its pre-image, we can always find $e$ an embedding of $\R^d$'s containing $e_i(\xi_i)$ and contained in all $e_i$. For instance by taking $e$ to small open discs around the points in the configuration.
        \end{enumerate}
    Then by Theorem \ref{thm:augmented semi-simplicial spaces and realisation}, the map $|\Dsum(M;P)_\bullet|\to \Csum(M;P)$ is a weak homotopy equivalence.
\end{proof}

The space $\Dsum(M;P)$ is equipped with an action of $\diff(M)$: for $\psi\in\diffM$, an embedding $e:\dcup{k}\R^d\hookrightarrow M$, and $\xi=(x_1,\dots,x_k;p_1,\dots,p_k)$ a configuration of points in $\dcup{k}\R^d$ with labels in $P$, we define $\phi\cdot [e,\xi]=[\phi\circ e,\xi]$.
It is simple to verify this action is well-defined and preserves the partial order in $\Dsum(M;P)$. It follows directly from the definition that the augmentation map $p:\Tsum(M;P)\to \Csum(M;P)$ is $\diffM$-equivariant.

Since the partial order in $\Dsum(M;P)$ is compatible with the $\diffM$-action, it induces a fibrewise partial order on $\Dsum(M;P)\times_\diffM \Emb(M,\Rinf)$ over $B\diffM$. Then the semi-simplicial nerve of the poset $\Dsum(M;P)\times_{\diffM} \Emb(M,\Rinf)$ is simply is the fibrewise semi-simplicial nerve $\Dsum(M;P)_\bullet\times_\diffM \Emb(M,\Rinf)$.

\begin{corollary}\label{cor:disc model for borel construction}
    The map $\Dsum(M;P)\to \Csum(M;P)$ induces a weak homotopy equivalence
        \[|\Dsum(M;P)_\bullet\pushout\diffM \Emb(M,\Rinf)|\longrightarrow \Csum(M;P)\parallelsum\diff(M).\]
\end{corollary}

The above follows directly from the definition of the definition of the partial order on $\Dsum(M;P)\times_\diffM \Emb(M,\Rinf)$ and  Proposition \ref{prop: weak equivalence with poset one}.

Corollary \ref{cor:disc model for borel construction} allows us to use a disc model for the space of configurations with partially summable labels when proving the decoupling in this setting. We finish this section by introducing another space of discs and configurations which will be used in the decoupling.

\begin{definition}\label{def: thin tubular configurations}
    Let $Z$ be a pointed $\GLd$-space and let $M$ be a smooth compact manifold of dimension $n>d$. The \emph{space of $d$-tubular configurations} in $M$ with labels in $Z$, denoted $\Tub^d(M;Z)$, is the quotient
        \[\left(\coprod\limits_{k\geq 0}  \Emb(\dcup{k} \R^d, M) \pushout{\Sigma_k\wr \GLdp} Z^k \right)\bigg/\sim\]
    where $(e_1,\dots,e_k;z_1,\dots,z_k)\sim(e_1,\dots,e_{k-1};z_1,\dots,z_{k-1})$
    whenever $z_k$ is the basepoint of $Z$, for $e_i:\R^d\hookrightarrow M$ and $z_i\in Z$. 
\end{definition}

The difference between the spaces $\Tub^d(M;Z)$ and $\Tub(M;Z)$ (Definition \ref{def: tubular configurations}) is that on the former we look at embedded discs of a lower dimension than the ambient manifold $M$.

\begin{lemma}\label{lemma: thin tubular configurations}
    Let $n>d$, and denote by $E_{d,n}$ denote the total space of the canonical $\GLdp$-bundle over the oriented Grassmanian $Gr^+(d,n)$. The inclusion of the origin $i:*\hookrightarrow\R^d$ induces a weak equivalence
        \[i^*:\Tub^d(\R^n;Z)\xrightarrow{\simeq} C(M;(E_{d,n})_+\pointedpushout{\GLdp}Z).\]
\end{lemma}

This result should be seen as an analogue of Lemma \ref{lemma: tubular configurations} in the setting of $d$-tubular configurations in $\R^n$.

\begin{proof}
    Recall that $\Emb(\dcup{k} \R^d, \R^n)\simeq \widetilde{C}_k(\R^n)\times (E_{d,n})^k$, where the map to $\widetilde{C}_k(\R^n)$ is induced by the inclusion of the origins. Then we get weak equivalences
        \[\Emb(\dcup{k} \R^d, M) \pushout{(\Sigma_k\wr \GLdp)} Z ^k\xrightarrow{\simeq}(\widetilde{C}_k(\R^n)\times (E_{d,n})^k)\pushout{(\Sigma_k\wr \GLdp)} Z^k.\]
    These respect the equivalence relations and therefore induce a map 
        \[i^*:\Dsum^d(\R^n;P)\xrightarrow{\simeq} C(M;(E_{d,n})_+\pointedpushout{\GLdp}Z).\]
    The proof then follows from the same arguments of \cite[Propositions 2.7 and 2.8]{Manthorpe-Tillmann}.
\end{proof}

We also need an analogue of Definition \ref{def: space of surrounded configurations} for the case of $d$-tubular configurations.

\begin{definition}\label{def:thin-discs}
    Let $M$ be a smooth compact manifold of dimension $n>d$.
    The \emph{space of $d$-surrounded configurations} in $M$ with summable labels in $P$, denoted $\Tsum^d(M;P)$, is the quotient
        \[\left(\coprod\limits_{k\geq 0}  \Emb(\dcup{k} \R^d, M) \pushout{\Sigma_k\wr \GLdp} \Csum(\dcup{k}\R^d;P) \right)\bigg/\sim\]
    where $(e:\dcup{m}\R^d\to  M,\xi)\sim(e':\dcup{n}\R^d\to  M,\xi')$ if there are injections $\underline{k}\hookrightarrow\underline{m}$ and $\underline{k}\hookrightarrow\underline{n}$ such that the induced inclusions $i_1:\dcup{k}\R^d\to \dcup{m}\R^d$ and $i_2:\dcup{k}\R^d\to \dcup{n}\R^d$ satisfy
        \begin{align*}
            \xi\subset Im\, i_1,\; \xi'\subset Im\, i_2 \text{ and } e\circ i_1=e'\circ i_2.
        \end{align*}
    We equip this space with a partial ordering by declaring $(e,\xi)<(e',\xi')$ if, for some representative of the classes, $e(\xi)=e'(\xi')$ and $Im\, e\supset Im\, e'$. We denote the semi-simplicial nerve of this poset by $\Tsum^d(M;P)_\bullet$. 
\end{definition}

\subsection{The decoupling theorem}\label{sec:decoupling map for colliding configurations}

In this section we prove the main theorem of this paper, which is a decoupling result for $\Csum(\Fgb;P)\parallelsum{\diffFgb}$. As in Section \ref{chap: configuration spaces}, we take as a model for $E\diffFgb$ the space $\Emb(\Fgb,\Rinf)$. For every $g$ and $b\geq0$ we have
    \begin{align*}
        \dcps:\Dsum(\Fgb;P)\times \Emb(\Fgb,\Rinf) &\longrightarrow \Emb(\Fgb,\Rinf) \times\Tsum^2(\Rinf;P)\\
        ((e:\dcup{k}\R^2\hookrightarrow \Fgb,\xi),f:\Fgb\hookrightarrow\Rinf) &\longmapsto (f,(f\circ e,\xi)).
    \end{align*}
This map is $\diffFgb$-equivariant with respect to the diagonal action on the domain and the action on $\Emb(\Fgb,\Rinf)$ on the target, and it preserves the poset structures. Hence it induces a map $\dcps$ fitting into the following diagram
    \[\begin{tikzcd}[column sep=1.2cm]
    	{{|\Dsum(\Fgb;P)_\bullet\pushout{\diffFgb} \Emb(\Fgb,\Rinf)|}} & {|B\diffFgb\times\Tsum^2(\Rinf;P)_\bullet|} \\
    	{{\Csum(\Fgb;P)\pushout{\diffFgb} \Emb(\Fgb,\Rinf)}} & {B\diffFgb \times |\Tsum^2(\Rinf;P)_\bullet|}
    	\arrow["\simeq"', from=1-1, to=2-1]
    	\arrow["\dcps", from=1-1, to=1-2]
    	\arrow["\simeq", from=1-2, to=2-2]
	\end{tikzcd}\]

\begin{theorem}\label{thm: decoupling summable configurations}
    The map 
    \[\dcps:|\Dsum(\Fgb;P)_\bullet\pushout{\diffFgb} \Emb(\Fgb,\Rinf)|\to |B\diffFgb\times\Tsum^2(\Rinf;P)_\bullet|\] 
    induces a homology isomorphism in degrees $\leq \frac{2}{3}g$.
\end{theorem}

Theorem \ref{thm: decoupling summable configurations} and Corollary \ref{cor:disc model for borel construction} imply Theorem \ref{main thm: decoupling summable labels}. 

The proof  of the above, will consist on showing that $\dcps$ is a level-wise homology isomorphism of semi-simplicial spaces in degrees $\leq \frac{2}{3}g$ and then show that this implies that the same holds on the geometric realisations, as done in \cite[Section 4]{MR3665002}. To do this, we will use the spectral sequence recalled in Section \ref{subseq: spectral sequence}, Lemma \ref{lemma: tubular configurations}, and Theorem \ref{thm: decoupling labelled configurations}, the decoupling result for the space of non-colliding configurations with labels.

Throughout the proof, it will be helpful to keep in mind Figure \ref{fig:visualising the homeomrphism}.

    \begin{figure}[h!]
        \centering
        \includegraphics[width=0.75\textwidth]{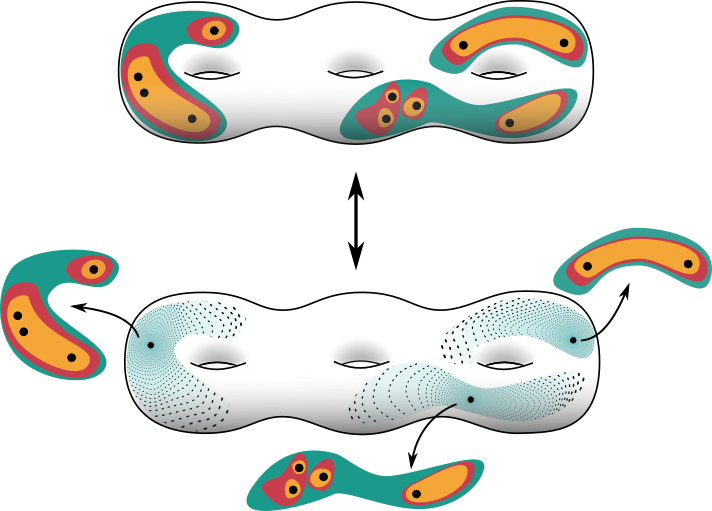}
        \caption{Correspondence of \eqref{eq:homeomorphism Dsum} between an element of $\Dsum(F_3;P)_2$ (top) and $\Tub(F_3,\mmm{Z}_2)$ (bottom). The dotted regions represent the embedded $\R^2$'s and the arrows indicate their labels.}
        \label{fig:visualising the homeomrphism}
    \end{figure}

\begin{proof}[Proof of Theorem \ref{thm: decoupling summable configurations}]
    We start by showing that 
        \[\Dsum(\Fgb;P)_\bullet\pushout{\diffFgb}\Emb(\Fgb,\Rinf) \to B\diffFgb\times\Tsum^2(\Rinf;P)_\bullet\]
    induces level-wise homology isomorphisms in degrees $\leq \frac{2}{3}g$.
    
    Let $\mmm{Z}_p$ be the subspace of $\Dsum(\R^2;P)_p$ consisting of those tuples $\overline{e}=((e_0,\xi_0),\dots,(e_p,\xi_p))$ with $e_0=\id$. 
    This space is pointed by the unique class on the empty configuration. 
    We show there is a $\diffFgb$-equivariant homeomorphism 
        \begin{align}\label{eq:homeomorphism Dsum}
            \Dsum(\Fgb;P)_p\cong \Tub(\Fgb;\mmm{Z}_p)
        \end{align}
    where $\Tub(\Fgb;\mmm{Z}_p)$ is the space of tubular configurations with labels in $\mmm{Z}_p$ (see Definition \ref{def: tubular configurations}). Let $[(e_0,\xi_0)<\dots<(e_p,\xi_p)]$ denote the an element in $\Dsum(\Fgb;P)_p$. Then, by definition $e_0(\xi_0)=\dots=e_p(\xi_p)$ and $Im\, e_0\supset\dots \supset Im\, e_p$. Denote by $i_j:\R^2\hookrightarrow \dcup{k}\R^2$ the map induced by the inclusion $\{j\}\hookrightarrow\{1,\dots,k\}$, for $1\leq j\leq k$.
    The map $\Dsum(\Fgb;P)_p\to \Dsum(\Fgb;\mmm{Z}_p)$ takes a sequence $((e_0,\xi_0),\dots,(e_p,\xi_p))$, to the class represented by the embedding $e_0$ and
    and the label associated to the $j$th component $\R^2\subset \dcup{k}\R^2$ given by 
        \[((id,\xi_0),((e_0\circ i_j)^{-1}\circ e_1,\xi_1)\dots,((e_0\circ i_j)^{-1}\circ e_p,\xi_p))\in \mmm{Z}_p.\]
    For intuition behind this homeomorphism, see Figure \ref{fig:visualising the homeomrphism}. It is simple to explicitly construct an inverse for this map and check this is a $\diffFgb$-equivariant 
    homeomorphism.
    
    By Lemma \ref{lemma: tubular configurations}, the inclusion of the origins defines $\diffFgb$-equivariant map $\Tub(\Fgb;\mmm{Z}_p)\to C(\Fgb;\mmm{Z}_p)$ which is a weak-equivalence of $\diffFgb$-spaces. Together with the homeomorphism \eqref{eq:homeomorphism Dsum} this implies that
        \begin{align}\label{eq:dsum and configurations}
            \Dsum(\Fgb;P)_p\pushout{\diffFgb} \Emb(\Fgb,\Rinf)\xrightarrow{\simeq} C(\Fgb;\mmm{Z}_p)\pushout{\diffFgb} \Emb(\Fgb,\Rinf).
        \end{align}
    
    Similarly, we now show that 
        \begin{align}\label{eq:homotopy equivalence for thin discs p}
            \Tsum^2(\Rinf;P)_p\xrightarrow{\simeq} C(\Rinf;(E\SO(2))_+\pointedpushout{\SO(2)}\mmm{Z}_p).
        \end{align}
    By the same argument as above, we get a homeomorphism
        \[\Tsum^2(\Rinf;P)_p\cong \Tub^2(\Rinf;\mmm{Z}_p)\]
    where $\Tub^2(\Rinf;\mmm{Z}_p)$ is the space of $d$-tubular configurations with labels in $\mmm{Z}_p$ as in Definition \ref{def: thin tubular configurations}. By Lemma \ref{lemma: thin tubular configurations}, the inclusion of the origins induces a weak homotopy equivalence $\Tub^2(\Rinf;\mmm{Z}_p)\simeq C(\Rinf;(E\SO(2))_+\wedge_{\SO(2)}\mmm{Z}_p)$ 
    which implies the homotopy equivalence \eqref{eq:homotopy equivalence for thin discs p}.

    Then we have a commutative diagram 
        \[\begin{tikzcd}[row sep=1cm]
        {\Dsum(\Fgb;P)_p\pushout{\diffFgb} \Emb(\Fgb,\Rinf)} \ar[r, "\simeq"] \ar[d, "(\dcps)_p"'] & {C(\Fgb;\mmm{Z}_p)\pushout{\diffFgb} \Emb(\Fgb,\Rinf)} \ar[d, "\tau\times\varepsilon"]\\
        {B\diffFgb\times\Tsum^2(\Rinf;P)_p} \ar[r, "\simeq"] & {B\diffFgb\times C(\Rinf;E\SO(2)_+\pointedpushout{\SO(2)}\mmm{Z}_p)}
    \end{tikzcd}\]
    where the top and bottom maps are weak equivalences by \eqref{eq:dsum and configurations} and \eqref{eq:homotopy equivalence for thin discs p}, respectively.
    The right-hand map induces homology isomorphisms in degrees $\leq\frac{2}{3}g$ by Theorem \ref{thm: decoupling labelled configurations} with $Z=\mmm{Z}_p$. Therefore so does the map $(\dcps)_p$.
    
    This implies that the map between the spectral sequences associated to the semi-simplicial spaces $X_\bullet=\Dsum(\Fgb;P)_\bullet\times_{\diffFgb}\Emb(\Fgb,\Rinf)$, and $Y_\bullet=B\diffFgb\times\Tsum^2(\Rinf;P)_\bullet$
        \[\begin{tikzcd}
            E^1_{p,q}=H_q(X_p) \ar[r, Rightarrow] \ar[d] & H_{p+q}(|X_\bullet|) \ar[d, "\dcps"] \\
            E'^1_{p,q}=H_q(Y_p) \ar[r, Rightarrow]  & H_{p+q}(|Y_\bullet|).
        \end{tikzcd}\]
    induces an isomorphism on the $E^1$-pages for all $q\leq\frac{2}{3}g$, and therefore the right-hand map is also an isomorphism in such degrees.    
\end{proof}

As in the case for labelled configuration spaces, the Decoupling Theorem for Summable Labels, allows us to deduce homological stability results.

\begin{corollary}\label{cor: homological stability for moduli of summable configurations}
    For $b\geq 1$, the map induced by gluing $F_{1,1}$ along the boundary
        \[\Csum(\Fgb;Z)\parallelsum\diffFgb\to \Csum(F_{g+1,b};Z)\parallelsum\diff(F_{g+1,b})\]
    induces a homology isomorphism in degrees $\leq \frac{2}{3}g$.
\end{corollary}

The proof follows from the same arguments as in the proof of Corollary \ref{cor: homological stability for moduli of labelled configurations}.

\subsection{Monoids of configurations on surfaces with partially summable labels}

In this section, we show that Theorem \ref{thm: decoupling summable configurations} implies a splitting for the group completion of the monoid of configuration on surfaces with partially summable labels, analogous to Corollary \ref{thm:group-completion}.

As in Section \ref{subsec: monoids of moduli of configuration spaces}, gluing two surfaces $\Fgo$ and $F_{h,1}$ along part of their boundary defines an associative multiplication
    \[B\diff(\Fgo)\times B\diff(F_{h,1})\to B\diff(F_{g+h,1}).\]

For $(P,0)$ a framed partial $2$-monoid with unit, the operation on diffeomorphism groups and spaces $E\diffFgo$ described above, together with the fixed identifications $F_{g+h,1}=\Fgo\#_\partial F_{h,1}$, induce an associative multiplication also on the Borel constructions
    \[\mu:\Csum(\Fgo;P)\parallelsum\diffFgo\times \Csum(F_{h,1};P)\parallelsum\diff(F_{h,1})\to \Csum(F_{g+h,1};P)\parallelsum\diff(F_{g+h,1}).\]
Analogous to the construction of Chapter \ref{chap: configuration spaces}, we denote the associated Borel construction by 
    \[M\Csum(P)_g=\Csum(\Fgo;P)\parallelsum \diff(\Fgo).\]
This multiplication makes $M\Csum(P)=\dcup{g\geq 0}M\Csum(P)_g$ into a topological monoid, which we refer to as the \emph{monoid of configurations with summable labels} in $P$. 

On the other hand, the poset $\Tsum^2(\Rinf;P)$ can be made into a partially ordered topological monoid, using the same strategy Segal used to define a topological monoid equivalent to $C(\Rinf,X)$, as we recalled in Section \ref{subsec: monoids of moduli of configuration spaces}.

\begin{corollary}\label{thm: group completion summable}
    For any path-connected framed partial $2$-monoid with unit $P$, there is a homotopy equivalence
        \[\Omega B(M\Csum(P))\simeq \Omega B\left(\dcup{g}B\diffFgo\right)\times \Omega B\,|\Tsum^2(\Rinf;P)_\bullet|.\]
\end{corollary}

The proof of the above result consists of constructing a zig-zag of monoids
    \begin{equation}\label{eq: zig-zag of monoids}
        \begin{tikzcd}[row sep=0.65cm, column sep=1.25cm]{{\coprod\limits_{g\geq0}|\Tsum(\Fgo;P)_\bullet\parallelsum\diffFgo|}} & {{\coprod\limits_{g\geq0}B\diffFgo\times|\Tsum^2(\Rinf;P)_\bullet|}} \\{{\phantom{\coprod\limits_{g\geq0}}M\Csum(P)\phantom{\coprod\limits_{g\geq0}}}}
        \arrow["\simeq"', from=1-1, to=2-1]
        \arrow["\dcps", from=1-1, to=1-2]
        \end{tikzcd}
    \end{equation}
The vertical arrow is induced by Corollary \ref{cor:disc model for borel construction} while the horizontal arrow is induced by the decoupling map. We describe the monoidal structures on the top spaces using a general construction for posets recalled in Section \ref{subsec: semi-simplicial nerve of a poset}.

\begin{proof}[Proof of Corollary \ref{thm: group completion summable}]
    The space ${\coprod\limits_{g\geq0}|\Tsum(\Fgo;P)_\bullet\parallelsum\diffFgo|}$ is homotopy equivalent to the geometric realisation of the semi-simplicial nerve of the poset
        \[{\coprod\limits_{g\geq0}\Tsum(\Fgo;P)\parallelsum\diffFgo}\]
    with partial order defined component-wise. 
    As before, gluing surfaces along part of their boundary makes $\dcup{g\geq0}\Tsum(\Fgo;P)\parallelsum\diffFgo$ into a partially ordered topological monoid. Moreover, one can verify that the augmentations $\Tsum(\Fgo;P)_\bullet\to \Csum(\Fgo;P)$ induce a map of monoids. Together with Lemma \ref{lemma: monoid on geom realisation}, this gives a map of monoids
        \begin{align}\label{eq: monoidal map 1}
        {\coprod\limits_{g\geq0}|\Tsum(\Fgo;P)_\bullet\parallelsum\diffFgo|}\to M\Csum(P).
        \end{align}
    
    On the other hand, the decoupling map gives a map of topological monoids
        \begin{align}\label{eq: monoidal map 2}
            {\coprod\limits_{g\geq0}|\Tsum(\Fgo;P)_\bullet\parallelsum\diffFgo|}\to{\coprod\limits_{g\geq0}B\diffFgo\times|\Tsum^2(\Rinf;P)_\bullet|} 
        \end{align}
    just as in Lemma \ref{lemma:group completion argument}.
    
    All that remains is to verify that the maps \eqref{eq: monoidal map 1} and \eqref{eq: monoidal map 2} induce homotopy equivalences on group completions. As the group completions are loop spaces, they are simple and, by Whitehead's theorem, it suffices to show the maps induce homology equivalences on the group completions. All monoids are homotopy commutative, hence the group completion theorem \cite{McDuff-Segal} can be applied. It implies that it is enough to show that the induced maps on limit spaces
        \begin{align}
        {|\Tsum(\Finf;P)_\bullet\parallelsum\diffFinf|}&\to \Csum(\Finf;P)\parallelsum\diffFinf\\
        {|\Tsum(\Finf;P)_\bullet\parallelsum\diffFinf|}&\to B\diffFinf\times|\Tsum^2(\Rinf;P)_\bullet|
        \end{align}
    are homology equivalences. The first map is a weak equivalence by Corollary \ref{cor:disc model for borel construction}, while the second map is a homology equivalence by Theorem \ref{thm: decoupling summable configurations} and Corollary \ref{cor: homological stability for moduli of summable configurations}.
\end{proof}


\bibliographystyle{amsalpha}
\bibliography{bibliography}

\bigskip
\end{document}